\documentclass{amsart}
\usepackage{amsfonts,amscd,amsthm,amsgen,amsmath,amssymb}
\usepackage[all]{xy}
\usepackage[vcentermath]{youngtab}
\newtheorem{theorem}{Theorem}[section]

\newtheorem{proposition}[theorem]{Proposition}

\theoremstyle{definition}
\newtheorem{definition}[theorem]{Definition}

\theoremstyle{remark}
\newtheorem{remark}[theorem]{Remark}

\numberwithin{equation}{section}

\begin{document}

\title{Variation of Hodge structure for generalized complex manifolds}
\author{David Baraglia}

\address{Mathematical sciences institute, The Australian National University, Canberra ACT 0200, Australia}


\email{david.baraglia@anu.edu.au}

\thanks{This work is supported by the Australian Research Council Discovery Project DP110103745.}


\date{\today}


\begin{abstract}
A generalized complex manifold which satisfies the $\partial \overline{\partial}$-lemma admits a Hodge decomposition in twisted cohomology. Using a Courant algebroid theoretic approach we study the behavior of the Hodge decomposition in smooth and holomorphic families of generalized complex manifolds. In particular we define period maps, prove a Griffiths transversality theorem and show that for holomorphic families the period maps are holomorphic. Further results on the Hodge decomposition for various special cases including the generalized K\"ahler case are obtained.
\end{abstract}

\maketitle


\section{Introduction}

Generalized complex structures, introduced by Hitchin \cite{hit} and further developed by Gualtieri \cite{gual} and Cavalcanti \cite{cav} are a hybrid of complex and symplectic geometry. They have gained popularity for their connections to string theory compactifications, supersymmetry and mirror symmetry. Aside from the unification of complex and symplectic geometry there are other features which point towards a deep connection between mirror symmetry and generalized complex geometry.

On a generalized complex manifold $M$ there is an analogue of the $\overline{\partial}$-operator for complex manifolds, which we also denote by $\overline{\partial}$. This gives rise to cohomology groups $H^k_{\overline{\partial}}(M)$ which are to be thought of as a kind of Hochschild homology since in the case that $M$ is an ordinary complex manifold one has $H^k_{\overline{\partial}}(M) = \bigoplus_{q-p=k} H^{p,q}(M)$. In certain cases the $\overline{\partial}$-cohomology gives rise to a decomposition of the cohomology of $M$, or in the presence of a non-trivial twisting $3$-form $H$, a decomposition in twisted cohomology (the relevant definitions are recalled in the paper)
\begin{equation}\label{hdc}
H^*(M,H) \otimes \mathbb{C} = \bigoplus_{k=-n}^n H^k_{\overline{\partial}}(M).
\end{equation}
More generally $H^*(M,H)\otimes \mathbb{C}$ is related to the $H^*_{\overline{\partial}}(M)$ groups through a spectral sequence. The decomposition (\ref{hdc}) occurs precisely when $M$ satisfies the generalized complex version of the $\partial \overline{\partial}$-lemma. In this case (\ref{hdc}) is a Hodge structure on twisted cohomology. This paper concerns aspects of these Hodge structures and in particular their behavior with respect to families of generalized complex structures.\\

Recall that the classical picture of mirror symmetry involves an isomorphism between A-model correlation functions on one side and B-model correlation functions on the mirror. The A-model correlation functions are related to quantum cohomology while the B-model is related to variation of Hodge structure. To explain the B-model further let $M$ be a compact Calabi-Yau manifold. We may define the Hochschild cohomology $HH^*(M)$ and Hochschild homology $HH_*(M)$ of $M$ by \cite{cal2}
\begin{align*}
HH^k(M) &= \bigoplus_{q+p=k} H^q( \wedge^p T^{1,0}M ) \\
HH_k(M) &= \bigoplus_{q-p = k} H^q( \wedge^{p,0} T^*M).
\end{align*}
Observe that $HH^*(M)$ is a graded ring and $HH_*(M)$ is a graded module for $HH^*(M)$. Since $M$ is a Calabi-Yau manifold there is a holomorphic volume form $\Omega$ which induces an isomorphism $\Omega : HH^k(M) \simeq HH_{n-k}(M)$. We can think of either $HH^*(M)$ or $HH_*(M)$ as the space of states in the topological $B$-model \cite{mor} and the correlation functions are essentially given by the module structure $HH^j(M) \otimes HH_k(M) \to HH_{j+k}(M)$. The special case
\begin{equation*}
H^1( T^{1,0}M) \otimes H^q( \wedge^{p,0} T^*M) \to H^{q+1}(\wedge^{p-1,0} T^*M)
\end{equation*}
describes infinitesimal variation of Hodge structure.

Kapustin and Li \cite{kapli} show how a generalized Calabi-Yau manifold gives rise to a topological field theory generalizing the A and B models. The ring $HH^*(M)$ is now given by Lie algebroid cohomology $H^*(L)$, where $L$ is the Lie algebroid associated to the generalized complex structure and $HH_*(M)$ is now given by the $\overline{\partial}$-cohomology $H^*_{\overline{\partial}}(M)$. We will see that the case
\begin{equation*}
HH^2(M) \otimes HH_k(M) \to HH_{k+2}(M)
\end{equation*}
similarly corresponds to infinitesimal variation of the Hodge structure (\ref{hdc}). We find therefore that variation of Hodge structure on a generalized Calabi-Yau manifold is related to the correlation functions in the topological field theory of Kapustin and Li. Speculatively we might conjecture that variations of Hodge structure are related to a quantum cohomology ring on a mirror generalized Calabi-Yau manifold.\\

In order to describe variations of Hodge structure for generalized complex manifolds we first need to define what we mean by a family of generalized complex structures. A simple definition would be to fix a manifold $M$ and closed $3$-form $H$ and consider a collection $J(t)$ of generalized complex structures that depend on a parameter $t \in B$. This definition seems overly restrictive since for instance we would like to allow for families defined over a punctured disc. Using Courant algebroids we find a more natural notion for a family, in fact we find there are two possible definitions which essentially correspond to smooth and holomorphic families. Section \ref{families} introduces the notion of families of exact Courant algebroids and this is used in Section \ref{smhofam} to define smooth families (Definition \ref{smoothfam}) and holomorphic families (Definition \ref{holofam}).\\

A generalized complex manifold $(M,J)$ gives rise to a complex Lie algebroid $L$. In \cite{gual} it was found that infinitesimal deformations of $J$ are described by $H^2(L)$, degree $2$ Lie algebroid cohomology of $L$. We call the class of such a deformation the {\em generalized Kodaira-Spencer class}. In Section \ref{gksc}, we reinterpret generalized Kodaira-Spencer classes in terms of our definition of smooth families.\\

Section \ref{varhodstr} represents the core of the paper. Here we consider how the Hodge decomposition (\ref{hdc}) varies in smooth and holomorphic families of compact generalized complex manifolds. Proposition \ref{ddbstab} shows that the existence of a Hodge decomposition is preserved under all sufficiently small deformations. For a family $X \to B$ of generalized complex structures, we find that the twisted cohomology of the fibers determines a vector bundle with flat connection, an analogue of the Gauss-Manin connection (Section \ref{twgmconn}). If a fiber $M_0 = \pi^{-1}(0)$ satisfies a Hodge decomposition (\ref{hdc}) then by Proposition \ref{ddbstab}, so do all sufficiently close fibers and the $\overline{\partial}$-cohomology groups $H^k_{\overline{\partial}}(M_t)$ define smooth subbundles of the bundle of twisted cohomology. As with variations of Hodge structures in complex geometry it is easier to work with the associated Hodge filtrations, which take the form:
\begin{equation*}
F^p H(M_t) = \dots \oplus H^{p-4}_{\overline{\partial}}(M_t) \oplus H^{p-2}_{\overline{\partial}}(M_t) \oplus H^p_{\overline{\partial}}(M_t).
\end{equation*}
Using these filtrations we may define periods maps (over a simply connected subset of the base):
\begin{equation*}
\mathcal{P}^{p} : B \to {\rm Grass}( f^p , H^{p+n+\tau}(M_0,H_0)_\mathbb{C} ),
\end{equation*}
where $f^p = {\rm dim}_{\mathbb{C}} F^p H(M_t)$, $M_0$ is $2n$-dimensional, $\tau$ is the parity (Section \ref{gcg}) and ${\rm Grass}(r,V)$ is the Grassmannian of complex $r$-dimensional subspaces of $V$. Proposition \ref{grtrans} shows that an analogue of Griffiths transversality holds: if $\nabla$ is the flat connection on the bundle of twisted cohomology groups and $e$ a vector field on the base, then $\nabla_e ( F^p H(M_t) ) \subseteq F^{p+2} H(M_t)$. The induced map
\begin{equation*}
F^p H(M_t) / F^{p-2} H(M_t) \to F^{p+2} H(M_t) / F^p H(M_t)
\end{equation*}
is shown to be the map $H^p_{\overline{\partial}}(M_t) \to H^{p+2}_{\overline{\partial}}(M_t)$ given by the Clifford action of the generalized Kodaira-Spencer class corresponding to the deformation in the direction $e$. In Proposition \ref{pmholo}, we use this to show that for holomorphic families of generalized complex manifolds the period maps are holomorphic.\\

In Section \ref{speccas}, we consider the Hodge decomposition and variations of Hodge structure in some special cases. For a symplectic manifold $(M,\omega)$, the existence of the Hodge decomposition is equivalent to the strong Lefschetz property. Section \ref{symplectic} studies this case and we show that the Hodge decomposition is completely determined by the action of $\omega$ on the cohomology of $M$. At the other extreme we consider generalized complex manifolds of complex type in Section \ref{cpxtype}. Note that this consists of a complex manifold $(M,I)$ together with a real closed $3$-form $H$ of type $(1,2) + (2,1)$. The Hodge decomposition is not the usual one for complex manifold, but rather a twisted version which incorporates $H$. We recall that there is a natural filtration $\{ W^j \}$ on twisted cohomology induced differential form degree. We show in Proposition \ref{mhs} that if the Hodge decomposition (\ref{hdc}) holds, then the Hodge structure together with the filtration $\{ W^i \}$ forms a mixed Hodge structure. This means that the Hodge decomposition induces Hodge decompositions on the quotients $W^j/W^{j+1}$. In the case $H=0$ we have $W^j/W^{j+1} = H^j(M,\mathbb{C})$ and we recover the usual Hodge decomposition for $M$. In Section \ref{gcym} we consider the case of generalized Calabi-Yau manifolds. For a compact generalized Calabi-Yau manifold which satisfies the $\partial \overline{\partial}$-lemma a result of Goto \cite{got} shows that all infinitesimal deformations are unobstructed and there is a smooth local moduli space $\mathcal{M}$ of generalized Calabi-Yau structures. As an application of the results on variation of Hodge structure, we show in Proposition \ref{pmimm} that the period map $\mathcal{P}^{-n} : \tilde{\mathcal{M}} \to \mathbb{P}( H^\tau(M , H)_\mathbb{C} )$ which sends a pure spinor $\rho$ to span of the twisted cohomology class $[\rho]$ is an immersion. Here $\tilde{\mathcal{M}}$ is a cover of $\mathcal{M}$ taken so that the period map is single-valued.\\

Finally in Section \ref{vhsgkm} we consider how the Hodge decomposition behaves for families of generalized K\"ahler manifolds. A generalized K\"ahler manifold $M$ has two commuting generalized complex structures $J_1,J_2$ and the $\partial \overline{\partial}$-lemma holds for both \cite{gualhod}, so there are two decompositions $\{ H^k_{\overline{\partial}_1}(M) \}$, $\{ H^k_{\overline{\partial}_2}(M)\}$ in twisted cohomology. Moreover, the two decompositions are compatible in that there is a decomposition
\begin{equation*}
H^*(M,H) \otimes \mathbb{C} = \bigoplus_{j,k} H^{j,k}_{\overline{\delta}_+}(M).
\end{equation*}
Here $\overline{\delta}_+$ is a certain differential on the space of forms on $M$ defined in Section \ref{gkg} and as subspaces of $H^*(M,H) \otimes \mathbb{C}$ we have $H^{j,k}_{\overline{\delta}_+}(M) = H^j_{\overline{\partial}_1}(M) \cap H^k_{\overline{\partial}_2}(M)$. The way this decomposition can vary in families $X \to B$ of generalized K\"ahler manifolds is essentially controlled by the generalized Kodaira spencer classes $\rho_i : TB \to H^2(L_i)$ for $i=1,2$ representing deformations of $J_1,J_2$. The requirement that $J_1,J_2$ commute however imposes compatibility conditions on $\rho_1,\rho_2$ which we determine in Proposition \ref{infdefgk}. On a generalized K\"ahler manifold there are decompositions $L_1 = L_1^+ \oplus L_1^-$, $L_2 = L_1^+ \oplus \overline{L_1^-}$, for certain complex Lie algebroids $L_1^{\pm}$. In Section \ref{lad} we show how such decompositions give the differential complexes $(\wedge^* L_1^* , d_{L_1})$, $(\wedge^* L_2^* , d_{L_2})$ the structure of bigraded differential complexes. In particular there exists natural maps $H^2(L_1) \to H^2(L_1^\pm)$, $H^2(L_2) \to H^2(L_1^+)$ and $H^2(L_2) \to H^2( \overline{L_1^-})$. The compatibility conditions on $\rho_1,\rho_2$ are that they coincide under the maps to $H^2(L_1^+)$ and they are related by conjugation under the maps $H^2(L_1) \to H^2(L_1^-)$, $H^2(L_2) \to H^2(\overline{L_1^-})$. In Section \ref{gkvhs} we interpret these conditions in terms of variation of Hodge structure on generalized K\"ahler manifolds.


\section{Smooth families in generalized geometry}\label{families}


\subsection{Courant algebroids and generalized geometry}

One may define various geometric structures on a smooth manifold in terms of structure on the tangent bundle. Generalized geometry works similarly except the tangent bundle is replaced with an exact Courant algebroid. In this section we review Courant algebroids and develop a suitable notion of smooth families in generalized geometry.\\

Courant algebroids were introduced in \cite{lwx} as a generalization of the bracket used by Courant \cite{cour} in the context of Hamiltonian systems. Courant algebroids can be defined in terms of a skew-symmetric bracket, the {\em Courant bracket}, or as we prefer a bracket which is rarely skew-symmetric called the {\em Dorfman bracket}. We follow \cite{roy} in defining Courant algebroids in terms of Dorfman brackets.
\begin{definition}
A {\em Courant algebroid} on a smooth manifold $M$ consists of a vector bundle $E \to M$, $\mathbb{R}$-bilinear bracket $[ \, , \, ] : \Gamma(E) \otimes \Gamma(E) \to \Gamma(E)$ called the {\em Dorfman bracket}, non-degenerate bilinear form $\langle \, , \, \rangle$ and a bundle map $\rho : E \to TM$ called the {\em anchor} such that for all $a,b,c \in \Gamma(E)$ and $f \in \mathcal{C}^\infty(M)$, we have
\begin{itemize}
\item[(C1)]{$[a,[b,c]] = [[a,b],c] + [b,[a,c]]$,}
\item[(C2)]{$\rho[a,b] = [\rho(a),\rho(b)]$,}
\item[(C3)]{$[a,fb] = \rho(a)(f)b + f[a,b]$,}
\item[(C4)]{$[a,b] + [b,a] = d \langle a,b \rangle $,}
\item[(C5)]{$\rho(a) \langle b,c \rangle = \langle [a,b],c \rangle + \langle b,[a,c] \rangle$,}
\end{itemize}
where $[\rho(a),\rho(b)]$ is the Lie bracket of vector fields and $d$ is the operator $d : \mathcal{C}^\infty(M) \to \Gamma(E)$ defined by $\langle df , a \rangle = \frac{1}{2}\rho(a)(f)$.\\

A Courant algebroid $E \to M$ is {\em exact} if the sequence
\begin{equation*}\xymatrix{
0 \ar[r] & T^*M \ar[r]^-{\frac{1}{2}\rho^*} & E \ar[r]^-\rho & TM \ar[r] & 0
}
\end{equation*}
is exact. Here $\rho^*$ is the transpose $T^*M \to E^*$ of $\rho$ followed by the identification of $E$ and $E^*$ using the pairing. 
\end{definition}
Let $E$ be an exact Courant algebroid. A section $s : TM \to E$ of the anchor will be called an {\em isotropic splitting} if $s(TM)$ is isotropic with respect to the pairing on $E$. It is clear that isotropic splittings exist for any exact Courant algebroid $E$ and that on choosing such a splitting we get an isomorphism $E = TM \oplus T^*M$ such that the pairing $\langle \, , \, \rangle$ is given by
\begin{equation*}
\langle X + \xi , Y + \eta \rangle = \frac{1}{2}( \xi(Y) + \eta(X) )
\end{equation*}
and the operator $d : \mathcal{C}^\infty(M) \to \Gamma(E)$ sends a function $f$ to the image of $df$ under the inclusion $T^*M \to E$.\\

Exact Courant algebroids over a manifold $M$ are classified by third degree cohomology with real coefficients, $H^3(M,\mathbb{R})$. Indeed given an exact Courant algebroid $E$ and a choice of isotropic splitting there is a closed $3$-form $H$ such that the Dorfman bracket on $E$ is given by \cite{sw} 
\begin{equation}\label{courb}
[X+\xi ,Y+\eta ]_H = [X,Y] + \mathcal{L}_X \eta - i_Y d \xi + i_X i_Y H.
\end{equation}
A change in isotropic splitting for $E$ changes $H$ by an exact term but the cohomology class $[ H ] \in H^2(M,\mathbb{R})$, called the {\em \v{S}evera class} of $E$ is independent of the splitting. Conversely, for any closed $3$-form $H$ the $H$-twisted Dorfman bracket given in (\ref{courb}) defines an exact Courant algebroid. Given two closed $3$-forms $H,H'$ the associated exact Courant algebroids are isomorphic if and only if $H$ and $H'$ represent the same class in $H^3(M,\mathbb{R})$. Thus exact Courant algebroids on $M$ are classified by $H^3(M,\mathbb{R})$. This classification is due to \v{S}evera \cite{sev}.\\

If $B$ is a $2$-form on $M$ we let $B$ act on $E = TM \oplus T^*M$ according to $B(X+\xi) = i_X B$. We also let $e^B$ denote the exponentiated action $e^B(X+\xi) = X + \xi + i_X B$. Such a transformation is called a $B$-shift. For any $2$-form $B$ we note that $e^B$ preserves the pairing $\langle \, , \, \rangle$ and the anchor in the sense that $\rho \circ e^B = \rho$. If $B$ is closed then we also see that $e^B$ preserves the $H$-twisted Dorfman bracket for any closed $H$. More generally for any $B$ we have an identity $e^B[a,b]_H = [e^Ba,e^Bb]_{H+dB}$.


\subsection{Families of exact Courant algebroids}\label{familiesexact}

Generalized geometry can be described as the study of geometric structures on exact Courant algebroids. Given a geometric structure it is natural to consider its deformations. We have found that the notion of a family of exact Courant algebroids gives a suitable framework in which to study deformation problems in generalized geometry. 

\begin{definition}
Let $\pi : M \to B$ be a locally trivial fiber bundle and $V = {\rm Ker}(\pi_*) \subseteq TM$ the vertical bundle. A Courant algebroid $E \to M$ on is called a {\em family of exact
Courant algebroids} over $B$ if the anchor $\rho : E \to TM$ factors through a map $E \to V$ and the induced sequence
\begin{equation*}
0 \to V^* \to E \to V \to 0
\end{equation*}
is exact.
\end{definition}

We now explain how a family of exact Courant algebroids $E \to M$ as defined above gives rise to an exact Courant algebroid on each fiber. Let $b \in B$ and $M_b = \pi^{-1}(b)$ be the fiber of $\pi$ over $b$. Let $i : M_b \to M$ be the inclusion. The restriction to $M_b$ of the vertical bundle of of $\pi$ is canonically isomorphic to the tangent bundle of $M_b$, so we have an exact sequence
\begin{equation*}
0 \to T^*M_b \to E|_{M_b} \to TM_b \to 0.
\end{equation*}
Let $a,b$ be sections of $E|_{M_b}$ over $M_b$. We can find smooth extensions $\tilde{a},\tilde{b}$ of $a,b$ to sections of $E$ over $M$. Now we define a Dorfman bracket $[ \, , \, ]_{E|_{M_b}}$ by setting $[a,b]_{E|_{M_b}} = ( [\tilde{a},\tilde{b}]_E )|_{M_b}$. To show that this is independent of the choice of smooth extensions we need only show that if $r,s$ are sections of $E$ and $s$ vanishes on $M_b$ then so do $[r,s]_E$ and $[s,r]_E$. First note that since $[r,s]_E + [s,r]_E = \rho^* d \langle r,s \rangle$, it suffices to show that $[r,s]_E$ vanishes on $M_b$. We see that $\rho( [r,s]_E) = [\rho(r),\rho(s)]$ and we know that $\rho(r),\rho(s)$ are vertical vector fields and $\rho(s)$ vanishes on $M_b$, so $[\rho(r),\rho(s)]$ also vanishes on $M_b$. Given any vertical vector field $W$ let $W'$ be a lift of $W$ to a smooth section of $E$. It now suffices to show that $\langle W' , [r,s]_E \rangle$ vanishes on $M_b$ for all such $W'$. But we have $\langle W' , [r,s]_E \rangle = \rho(r) \langle W' , s\rangle - \langle [r,W'] , s\rangle$ and the claim follows.\\

We note that the bundle $E|_{M_b}$ has in addition to the inherited bracket $[ \, , \, ]_{E|_{M_b}}$ a natural anchor map $E|_{M_b} \to TM_b$ and pairing $\langle \, , \, \rangle$ both defined by restriction from $E$. It is straightforward now to see that this structure makes $E|_{M_b}$ into an exact Courant algebroid on $M_b$.\\

Next we show that an exact Courant algebroid $F \to M$ on $M$ gives rise to a family of exact Courant algebroids over $B$ in a natural way. We have subbundles $A^\perp \subseteq A \subseteq F$ defined as follows. $A$ is the kernel of the composition $F \buildrel \rho \over \to TM \buildrel \pi_* \over \to \pi^*(TB)$ and $A^\perp$ is the annihilator of $A$. We note that $A,A^\perp$ are subbundles of $F$ since $\pi_*$ has constant rank and that $A^\perp \subseteq A$, indeed it is straightforward to see that $A^\perp$ can be identified with the image of $\pi^*(T^*B)$ under the inclusion $\rho^* : T^*M \to F$.
\begin{proposition}
The space $\Gamma(A)$ of sections of $A$ is a subalgebra of $\Gamma(F)$ and the sections $\Gamma(A^\perp)$ of $A^\perp$ is a two-sided ideal in $\Gamma(A)$. The quotient bundle $E = A/A^\perp$ with the induced bracket on $\Gamma(A/A^\perp)$ is a family of exact Courant algebroids on $M$. If $F$ has \v{S}evera class $H \in H^3(M,\mathbb{R})$ then the exact Courant algebroid on the fiber $M_b = \pi^{-1}(b)$ induced by $E$ has \v{S}evera class $H|_{M_b} \in H^3(M_b,\mathbb{R})$.
\end{proposition}
\begin{proof}
Let $a,b \in \Gamma(A)$, so $\rho(a),\rho(b)$ are vertical vector fields. Thus $\rho( [a,b]_F ) = [\rho(a),\rho(b)]$ is also vertical and $[a,b]_F \in \Gamma(A)$. Suppose now that $a \in \Gamma(A)$ and $b \in \Gamma(A^\perp)$. We must show that $[a,b]_F$ and $[b,a]_F$ take values in $A^\perp$. Note however that $[a,b]_F + [b,a]_F = \rho^* d \langle a , b \rangle = 0$, so it suffices to consider just $[a,b]_F$. We have $\rho( [a,b]_F ) = [\rho(a),\rho(b)] = 0$ since $\rho(b) = 0$. Thus $[a,b]_F$ is a $1$-form. Let $W$ be a vertical vector field and $W'$ a lift to a section of $F$ (in particular $W'$ is a section of $A$). It suffices to show that for all such $W,W'$ we have $\langle W' , [a,b]_F \rangle = 0$. Now $\langle W' , [a,b]_F \rangle = \rho(a)\langle W' , b \rangle - \langle [a,W']_F , b \rangle = \rho(a) ( b(W)) - b( [\rho(a),\rho(W)]) = 0$, since $b$ is in the annihilator of $A$.\\

We have shown that the bundle $E = A/A^\perp$ has a natural bracket on its space of sections. In addition the restriction to $A$ of the pairing $\langle \, , \, \rangle$ on $F$ descends to a non-degenerate pairing on $A/A^\perp$. Similarly the anchor map $\rho : F \to TM$ restricts to a map $A \to V$ which then descends to a map $A/A^\perp \to V$. It is straightforward to see that all of the induced structure makes $E$ into a Courant algebroid. We also see that the sequence $0 \to V^* \to E \to V \to 0$ is exact, so $E$ is a family of exact Courant algebroids over $B$.\\

Let $b \in B$ and $\pi^{-1}(b) = M_b$ the fiber over $b$. To determine the \v{S}evera class of $E|_{M_b}$ let us first choose an isotropic splitting $s : TM \to F$ for $F$.
Then there is a closed $3$-form $H \in \Omega^3(M)$ such that for all vector fields $X,Y$ on $M$ we have
\begin{equation}\label{splittinge}
[sX,sY]_F - s[X,Y] = i_Y i_X H.
\end{equation}
The cohomology class of $H$ in $H^3(M,\mathbb{R})$ is the \v{S}evera class of $F$. Observe that the restriction $s|_V$ of $s$ to the vertical tangent bundle $V$ maps into $A$ and by factoring out $A^\perp$ we get an induced isotropic splitting $\tilde{s} : V \to A/A^\perp$. If $X,Y$ are vertical vector fields, equation (\ref{splittinge}) becomes
\begin{equation*}
[\tilde{s}X,\tilde{s}Y]_E - \tilde{s}[X,Y] = i_Y i_X H \; {\rm mod}(A^\perp).
\end{equation*}
Restricting to the fiber $M_b$ we immediately see that the \v{S}evera class of $E|_{M_b}$ is just $[H]|_{M_b}$.
\end{proof}
The family of exact Courant algebroids $E = A/A^\perp$ obtained from $F$ will be called the {\em fiberwise reduction} of $F$. This reduction procedure is related to the reduction of Courant algebroids in \cite{bcg} (in particular to Lemma 3.7).\\

We now turn to the problem of finding a general classification for families of exact Courant algebroids. In fact, the classification follows almost immediately from the classification of
$AV$-Courant algebroids. $AV$-Courant algebroids, defined in \cite{lb}, are a kind of generalization of exact Courant algebroids involving a Lie algebroid $A$ and an
$A$-module $V'$, where we use a prime to avoid a clash in notation. If we take $A$ to be the Lie algebroid of the vertical tangent bundle $A = V = {\rm Ker}(\pi_*)$ and $V' = \mathbb{R}$ the trivial line bundle then an $AV'$-Courant algebroid is precisely what we call a family of exact Courant algebroids over $B$. It follows from \cite{lb} that isomorphism classes of families of exact Courant algebroids are classified by $H^3(V,\mathbb{R})$, degree $3$ Lie algebroid cohomology for $V$ with values in $\mathbb{R}$.

Let $d^V : \Gamma( \wedge^k V^*) \to \Gamma( \wedge^{k+1} V^*)$ be the Lie algebroid differential for $V$ (with values in the trivial module $\mathbb{R}$). This defines a differential graded complex $(\Gamma( \wedge^k V^*) , d^V)$ and by definition the Lie algebroid cohomology $H^k(V,\mathbb{R})$ is the degree $k$ cohomology of this complex. Let $h \in H^3(V,\mathbb{R})$. An explicit representative for the corresponding family of exact Courant algebroids is obtained as follows. First choose a representative $H \in \Gamma(\wedge^3 V^*)$ for $h$, so in particular $H$ is $d^V$-closed. Now let $E = V \oplus V^*$ and define a Dorfman bracket $[ \, , \, ]_H$ by
\begin{equation*}
[ X + \xi , Y + \eta]_H = [X,Y] + \mathcal{L}_X \eta - i_Y d^V\xi + i_Y i_X H.
\end{equation*}
In the above equation $\mathcal{L}_X$ denotes the Lie derivative in the sense of Lie algebroids, thus $\mathcal{L}_X \eta = i_X d^V \eta + d^V i_X \eta$. Observe that $E$ has a natural non-degenerate bilinear form given by pairing $V$ and $V^*$ and a natural anchor map $E \to V$ which is just the projection to $V$. It is straightforward to see that this structure makes $E$ into a family of exact Courant algebroids. From \cite{lb} we know that every family of exact Courant algebroids is isomorphic to one of this form for some unique class $h \in H^3(V,\mathbb{R})$. In analogy with the case of exact Courant algebroids we call the class $h \in H^3(V,\mathbb{R})$ the \v{S}evera class of $E$.\\

Recall that an exact Courant algebroid $F$ on $M$ gives rise to a family $E = A / A^\perp$ of exact Courant algebroids over $B$. Suppose that $F$ has \v{S}evera class $h \in
H^3(M,\mathbb{R})$. The inclusion $V \to TM$ is a morphism of Lie algebroids and so defines a natural restriction map $H^3(M,\mathbb{R}) \to H^3(V,\mathbb{R})$. It is almost immediate that the \v{S}evera class of the family $E$ is just the image of $h$ in $H^3(V,\mathbb{R})$.

In general the map $H^3(M,\mathbb{R}) \to H^3(V,\mathbb{R})$ is neither injective nor surjective so we would like to clarify the relation between these two groups. For simplicity we will assume the fibers of $M$ are compact. Choose a splitting for the short exact sequence $0 \to V \to TM \to \pi^*(TB) \to 0$. We then have an isomorphism $TM = \pi^*(TB) \oplus V$. There is an induced bi-grading on differential forms where the bundle of degree $(p,q)$-forms, denoted $\wedge^{p,q} T^*M$ is defined to be $\wedge^p \pi^*(T^*B) \otimes \wedge^q V^*$. Similarly let $\Omega^{p,q}(M)$ denote the sections of $\wedge^{p,q} T^*M$. Next introduce a filtration $F^k \wedge^n T^*M = \oplus_{p \ge k} \wedge^{p,n-p} T^*M$ and $F^k \Omega^n(M) = \oplus_{p \ge k} \Omega^{p,n-p}(M)$. We note that $d F^k \Omega^n(M) \subseteq F^k \Omega^{n+1}(M)$, so there is a corresponding spectral sequence converging to the cohomology of $M$. The $E_1$ term of this spectral sequence is obtained by taking cohomology in the vertical direction. More precisely observe that $\wedge^p \pi^*(TB)$ is naturally a module for the Lie algebroid $V$. Indeed we let a vertical vector field $W$ act on $\omega \in \Omega^{p,0}(X)$ by defining $W(\omega) = i_W d\omega$. Note that since $i_W \omega = 0$ this coincides with the Lie derivative of $\omega$ along $W$, so if $W_1,W_2$ are vertical vector fields then
\begin{equation*}
W_1 (W_2(\omega)) - W_2 (W_1(\omega)) = \mathcal{L}_{W_2} \mathcal{L}_{W_1}(\omega) - \mathcal{L}_{W_1} \mathcal{L}_{W_2}(\omega) = \mathcal{L}_{[W_1,W_2]} \omega = [W_1,W_2]\omega
\end{equation*}
as required. The graded differential complex for $V$ with values in $\Omega^{p,0}(M)$ is just the complex
\begin{equation*}
\Omega^{p,0} \buildrel d^V \over \longrightarrow \Omega^{p,1} \buildrel d^V \over \longrightarrow \Omega^{p,2} \buildrel d^V \over \longrightarrow \cdots
\end{equation*}
where $d^V$ is the vertical projection of the exterior derivative. It follows that the $E_1$-stage of the spectral sequence is given by
\begin{equation*}
E_1^{p,q} = H^q( V , \wedge^{p,0} T^*M).
\end{equation*}
The natural map $H^q(M,\mathbb{R}) \to H^q(V,\mathbb{R})$ is the composition $H^q(M,\mathbb{R}) \to E_\infty^{0,q} \to E_1^{0,q} = H^q(V,\mathbb{R})$. We remark that the $E_2$-stage of this spectral sequence is just the Serre spectral sequence for the fibration $\pi : M \to B$ \cite{grha}, namely $E_2^{p,q} = H^p(B , R^q \pi_* \mathbb{R})$. To explain the relation between the $E_1$ and $E_2$ stages we need to observe that $H^q(V , \wedge^{p,0} T^*M)$ is isomorphic to $\Omega^p(B , R^q \pi_* \mathbb{R} )$, that is the space of $p$-forms on the base with values in the flat vector bundle $R^q \pi_* \mathbb{R}$. The main step is to show that if $\omega$ is an $(0,q)$-form on $M$ which is $d^V$-closed and its restriction to each fiber is exact, then there is an $(0,q-1)$-form $\alpha$ such that $\omega = d^V \alpha$. This can be achieved by first using a partition of unity on the base to localize the problem and then using standard Hodge theory techniques.

Applying this in particular to degree $3$ cohomology we have that $H^3(V,\mathbb{R})$ is isomorphic to the smooth sections of the flat bundle $R^3 \pi_* \mathbb{R}$. On the other hand the restriction of a class in $H^3(M,\mathbb{R})$ to $H^3(V,\mathbb{R})$ maps into the space of {\em constant} sections of $R^3 \pi_* \mathbb{R}$. In general not every constant section of $R^3 \pi_* \mathbb{R}$ comes from a class in $H^3(M,\mathbb{R})$, since there are higher order differentials in the spectral sequence. However, if we suppose the base $B$ is contractible then $H^3(M,\mathbb{R})$ is precisely the space of constant sections of $R^3 \pi_* \mathbb{R}$. Note further in this case that $R^3 \pi_* \mathbb{R}$ is the trivial bundle with fiber isomorphic to $H^3(M_0,\mathbb{R})$, where $M_0$ is a fiber of $M$. To summarize, when $B$ is contractible we have an isomorphism $H^3(V,\mathbb{R}) \simeq \mathcal{C}^\infty(B) \otimes H^3(M_0,\mathbb{R})$ and $H^3(M,\mathbb{R}) \simeq H^3(M_0,\mathbb{R})$, the constant functions in $\mathcal{C}^\infty(B) \otimes H^3(M_0,\mathbb{R})$. In particular we have shown the following.
\begin{proposition}
Let $\pi : M \to B$ be a fiber bundle with compact fibers over a contractible base and let $E \to M$ be a family of exact Courant algebroids over $B$. Taking \v{S}evera classes of the fibers of $M$ defines a map $S : B \to H^3(M_0,\mathbb{R})$, where $M_0$ is a fiber of $M$. Then $E$ is the fiberwise reduction of an exact Courant algebroid on $M$ if and only if $S$ is constant.
\end{proposition}


\subsection{The twisted Gauss-Manin connection}\label{twgmconn}

On a smooth manifold $M$ consider the bundle $S = \wedge^* T^*M$ of forms on $M$ of mixed degree. Sections of $S$ will simply be called forms. There is a natural $\mathbb{Z}_2$-grading on $S$, namely $S^0 = \wedge^{ev}T^*M$ is the bundle of forms of even degree, $S^1 = \wedge^{odd} T^*M$ the odd degree forms and we refer to sections of $S^0$/$S^1$ as even/odd forms. If $\alpha$ is a form of degree $k$ we often write $(-1)^\alpha$ in place of $(-1)^k$. Since this only depends on the mod $2$ degree we can similarly introduce sign factors $(-1)^\alpha$ for any even or odd form.

Let $H$ be a closed $3$-form on $M$. Associated to $H$ is a twisted differential $d_H : \Gamma(S^i) \to \Gamma(S^{i+1})$ which is defined by $d_H(\alpha) = d\alpha + H \wedge \alpha$. Observe that $(d_H)^2 = 0$, so we have a $\mathbb{Z}_2$-graded complex $( \Gamma(S^*) , d_H)$. The cohomology of this complex is denoted $H^*(M,H)$ and called the {\em twisted cohomology} of $(M,H)$. For any $2$-form $B$, let $e^B : S^i \to S^i$ be given by
\begin{equation*}
e^B \alpha = \alpha + B\wedge \alpha + \frac{1}{2} B \wedge B \wedge \alpha + \dots
\end{equation*}
One sees that $d_H \circ e^B = e^B \circ d_{H+dB}$, so that $e^B$ descends to an isomorphism $e^B : H^*(M,H+dB) \to H^*(M,H)$ of twisted cohomology groups. In particular we see that if $H_1,H_2$ are closed $3$-forms representing the same class in de Rham cohomology, then the twisted cohomology groups $H^*(M,H_1)$,$H^*(M,H_2)$ are isomorphic. Note however that the isomorphism $e^B : H^*(M,H_1) \to H^*(M,H_2)$ depends on the choice of $2$-form such that $H_1 = H_2 + dB$.\\

Let $\pi : M \to B$ be a fiber bundle and $E \to M$ a smooth family of exact Courant algebroids over $B$. The \v{S}evera classes of the fibers determines an element $S \in \Gamma(B,R^3 \pi_* \mathbb{R})$, that is a smooth section of the bundle of degree $3$ fiber cohomology. For $b \in B$ let $M_b$ denote the fiber of $M$ over $b$. In general the dimension of the twisted cohomology groups $H^*(M_b , S(b) )$ will vary with $b$, so that it is not even a vector bundle on $B$. To get smoothly varying twisted cohomology groups we will generally assume that the \v{S}evera class $S$ is locally constant. Note that when $S(b)$ is the Dixmier-Douady class of a gerbe on $M_b$ we have that $S(b)$ is integral. In this case $S$ is necessarily locally constant.

Assuming now that the \v{S}evera class is constant we know that for sufficiently small open subsets in the base $U \subseteq B$ we have an exact Courant algebroid $F \to
\pi^{-1}(U)$ such that $E|_{U}$ is the fiberwise reduction of $F$. Moreover the \v{S}evera class $H \in H^3(\pi^{-1}(U) , \mathbb{R})$ coincides with $S$. Although it is not
strictly necessary, it will simplify matters to assume that $E$ is the fiberwise reduction of a globally defined exact Courant algebroid $F \to M$. Let $H \in H^3(M,\mathbb{R})$ be the \v{S}evera class of $F$. We have a $\mathbb{Z}_2$-graded presheaf on $B$ which assigns to an open subset $U \subseteq B$ the group
$H^*(\pi^{-1}(U),H|_{\pi^{-1}(U)})$. Let $\mathcal{H}^*$ denote the sheaf associated to this presheaf.
\begin{proposition}
Let $B$ be connected. The sheaf $\mathcal{H}^*$ is a local system with coefficient group isomorphic to $H^*(M_0,H|_{M_0})$, the twisted cohomology of a fiber $M_0$ of $M$.
\end{proposition}
\begin{proof}
Consider a contractible subset $U \subseteq B$. The restriction $M|_U = \pi^{-1}(U)$ of $M$ over $U$ admits a trivialization $M|_U = U \times M_0$, where $M_0$ is a fiber of $M$. We have that $H^3(M|_U , \mathbb{R}) = H^3(M_0,\mathbb{R})$, so there exists a class $H_0 \in H^3(M_0,\mathbb{R})$ such that $H|_{M|_U} = H_0$. We claim that the restriction $H^*(M|_U,H_0) \to H^*(M_0,H_0)$ is an isomorphism. Given this the proposition follows since $B$ is locally contractible. To see that the restriction $H^*(M|_U,H_0) \to H^*(M_0,H_0)$ is an isomorphism recall first that for any pair $(M,H)$ consisting of a manifold $M$ and closed $3$-form $H$ on $M$, there is a spectral sequence for twisted cohomology which arises from the filtration $F^p \Omega(M) = \oplus_{k \ge p} \Omega^ k(M)$ on differential forms. We have that $E_2^* =H^*(M,\mathbb{R})$ (where the grading is the $\mathbb{Z}_2$-grading by form degree). If $f : X \to Y$ is a smooth map between manifolds $X,Y$ and there are closed $3$-forms $H_X,H_Y$
such that $H_X = f^*(H_Y)$ then we get an induced morphism between the spectral sequences for $(Y,H_Y)$ and $(X,H_X)$. At the $E_2$-stage the morphism is just the pullback $f^* : H^*(Y,\mathbb{R}) \to H^*(X,\mathbb{R})$. In the case at hand we set $X = M_0$, $Y = U \times M_0$, $H_X = H_0$ and $H_Y = H|_{M|_U}$. Since we get an isomorphism at the $E_2$-stage the twisted cohomology groups are isomorphic.
\end{proof}

\begin{definition}
The vector bundle with flat connection corresponding to $\mathcal{H}^*$ will be denoted $(H^*,\nabla)$ and the flat connection $\nabla$ called the {\em twisted Gauss-Manin connection}.
\end{definition}
We give here a simple local description of the twisted Gauss-Manin connection. Let $\pi : M \to B$ be a locally trivial fiber bundle with compact oriented fibers and suppose we are given a section $h$ of $R^3 \pi_* \mathbb{R}$, the \v{S}evera class of the fibers. For any $b \in B$ we may choose an open subset $U \subseteq B$ containing $b$ such that there is a local trivialization $\pi^{-1}(U) = U \times M_0$ and a closed $3$-form $H \in \Omega^3(U \times M_0)$ such that the restriction of $H$ to the fiber $M_u = \{u\} \times M_0$ is a representative for $h(u)$. Suppose $\tilde{s}$ is a differential form on $U \times M_0$ such that:
\begin{itemize}
\item{$\tilde{s}$ contracted with a vector field on $U$ is trivial,}
\item{the restriction $\tilde{s}|_{M_u}$ of $\tilde{s}$ to $M_u$ is $d_H$-closed.}
\end{itemize}
Then $\tilde{s}$ determines a local section $s$ of $H^*$ by setting $s(u) = [ \tilde{s}|_{M_u}]$. Note also that any section of $H^*$ over $U \times M_0$ can be represented in this manner. Let $X$ be a vector field on $U$ and view $X$ as a vector field on the product $U \times M_0$. We claim that the twisted Gauss-Manin connection can be expressed as follows:
\begin{equation}\label{tgmformula}
(\nabla_X s)(u) = [ (i_X d_H \tilde{s}) |_{M_u}].
\end{equation}
One first shows that (\ref{tgmformula}) is a well-defined connection for sections of $H^*$ over $U \times M_0$. Next let $a$ be a fixed class in $H^*(M_0,H|_{M_0})$. By taking $U$ to be contractible we may assume that the restriction map $H^*(U\times M_0 , H) \to H^*(M_0 , H|_{M_0})$ is an isomorphism. Therefore we may choose a $d_H$-closed form $\tilde{a}$ on $U \times M_0$ representing $a$. Under the twisted Gauss-Manin connection $u \mapsto [\tilde{a}|_{M_u}]$ corresponds to the closed section which takes the value $a$ at $u=0$. On the other hand this agrees with (\ref{tgmformula}) since $d_H \tilde{a} = 0$. Therefore (\ref{tgmformula}) coincides with the twisted Gauss-Manin connection, since they have the same constant sections.\\

We introduce a natural pairing on twisted cohomology which is preserved by the twisted Gauss-Manin connection. First we define an involution $\sigma : \wedge^* T^*M \to \wedge^* T^*M$ on the bundle of forms on $M$ by setting $\sigma(\alpha) = (-1)^{k(k-1)/2}\alpha$, where $\alpha$ has degree $k$. One sees that for all $\alpha$ and all closed $3$-forms $H$
\begin{equation*}
\sigma( d_H ( \sigma( \alpha))) = (-1)^\alpha d_{-H} \alpha.
\end{equation*}
In particular $\sigma$ descends to an isomorphism $\sigma : H^*(M,H) \to H^*(M,-H)$ of twisted cohomology groups.\\

Suppose that $M$ is $n$-dimensional. The {\em Mukai pairing} is a bilinear pairing $\langle \, , \, \rangle$ on the bundle of forms on $M$ with values in the determinant bundle $\wedge^n T^*M$. It is given by the following expression:
\begin{equation*}
\langle \alpha , \beta \rangle = [ \alpha \wedge \sigma(\beta) ]^n,
\end{equation*}
where $[\omega]^n$ denotes the degree $n$ part of $\omega$. Suppose now $M$ is compact, oriented and that $H$ is a closed $3$-form. For $d_H$-closed forms $\alpha,\beta$ we define $(\alpha,\beta) \in \mathbb{R}$ as follows:
\begin{equation*}
( \alpha , \beta ) = \int_M \langle \alpha , \beta \rangle.
\end{equation*}
We claim that $(\alpha,\beta)$ depends only on the twisted cohomology classes of $\alpha,\beta$. Indeed this follows since for any forms $\omega,\gamma$ we have
\begin{equation*}
( d_H \omega , \gamma) = (-1)^n(\omega , d_H \gamma).
\end{equation*}
Alternatively a more cohomological interpretation of $( \, , \, )$ is as follows. For any two closed $3$-forms $H_1,H_2$ the wedge product naturally induces maps $\wedge : H^i(M,H_1) \otimes H^j(M,H_2) \to H^{i+j}(M,H_1+H_2)$. Let $[M]$ denote the fundamental class of $M$. Then if $\alpha,\beta$ are $d_H$-closed and $[\alpha],[\beta]$ the corresponding twisted cohomology classes we see that
\begin{equation*}
(\alpha,\beta) = (  [\alpha]  \wedge \sigma [\beta] ) [M],
\end{equation*}
where we have used the fact that $[\alpha] \wedge \sigma[\beta]$ is a class in ordinary cohomology, so it can be evaluated on $[M]$.\\

Suppose that $\pi : M \to B$ is a fiber bundle with compact oriented fibers and $H$ a closed $3$-form on $M$. We have the flat bundle $H^*$ with fiber over $b \in B$ given by the twisted cohomology group $H^*(M_b , H_b)$, where $M_b = \pi^{-1}(b)$ and $H_b = H|_{M_b}$. We have just seen that we can define a pairing $( \, , \, ) : H^*(M_b,H_b) \otimes H^*(M_b , H_b) \to \mathbb{R}$ on twisted cohomology. This gives the bundle $H^*$ a natural bilinear form $Q : H^* \otimes H^* \to \mathbb{R}$. It is not hard to see that $Q$ depends smoothly on $b \in B$, in fact we can show something much stronger.
\begin{proposition}
Let $a,b$ be constant locally defined sections of $H^*$. Then $Q(a,b)$ is constant. Therefore $Q$ defines a covariantly constant bilinear form on $H^*$.
\end{proposition}
\begin{proof}
We may restrict to a contractible open set $U \subseteq B$ in the base. Recall that the inclusion of a fiber $i : M_b \to \pi^{-1}(U)$ for $b \in U$ induces an isomorphism in twisted cohomology. The idea now is to define a pairing on $H^*( \pi^{-1}(U) , H)$ that coincides with the pairing on $H^*(M_b , H_b)$ under $i$. The difficulty is that $\pi^{-1}(U) $ will not be compact, so we can not define a pairing on $H^*( \pi^{-1}(U) , H)$ by integrating over $\pi^{-1}(U)$. One remedy would be to introduce twisted cohomology with compact support, but there is an even simpler alternative. Define a bilinear pairing $\lambda : H^*( \pi^{-1}(U) , H) \otimes H^*( \pi^{-1}(U) , H) \to \mathbb{R}$ on $H^*( \pi^{-1}(U) , H)$ as follows: for $a,b \in H^*( \pi^{-1}(U) , H)$ we set
\begin{equation*}
\lambda(a,b) = (a \wedge \sigma(b) )[M_b],
\end{equation*}
where $[M_b]$ is the fundamental class of the fiber $M_b$ in $H_*(\pi^{-1}(U),\mathbb{R})$. We observe that under the pullback isomorphism $i^* : H^*( \pi^{-1}(U) , H) \to H^*(M_b , H_b)$, $\lambda$ is identified with the pairing $( \, , \, ) = Q_b$. Finally note that that for all $b \in U$ the fiber $M_b$ above $b$ defines the same homology class $[M_b]$ in $H_*(\pi^{-1}(U),\mathbb{R})$. Therefore the pairing $\lambda$ restricts to $Q_b$ for all $b \in B$. This shows that $Q$ is given by a locally constant pairing.
\end{proof}


\section{Families of generalized complex structures}\label{holofamilies}

\subsection{Generalized complex geometry}\label{gcg}

Let $E \to M$ be a Courant algebroid on $M$ (exact or not). A {\em generalized almost complex structure} on $E$ is an endomorphism $J : E \to E$ such that $J^2 = -1$ and $\langle Ja , Jb \rangle = \langle a , b \rangle$ for all $a,b$. We may decompose $E$ into the $\pm i$-eigenspaces of $J$:
\begin{equation*}
E \otimes \mathbb{C} = L \oplus \overline{L},
\end{equation*}
where $L$ is the $+i$-eigenspace and $\overline{L}$ is the complex conjugate of $L$, which is then the $-i$-eigenspace. We say that $J$ is {\em integrable} if $L$ is closed under the Dorfman bracket. An integrable generalized almost complex structure is then called a {\em generalized complex structure}. Note that $L,\overline{L}$ are isotropic subspaces and it follows that the restriction of the Dorfman bracket gives $L,\overline{L}$ the structure of complex Lie algebroids.\\

Suppose that $E \to M$ is the exact Courant algebroid on $E$ with \v{S}evera class $h \in H^3(M,\mathbb{R})$. Given a closed $3$-form $H$ representing $h$ we can take $E$ to be $TM \oplus T^*M$ with the $H$-twisted Dorfman bracket. Let $S$ be the bundle of differential forms. There is a Clifford action of $E$ on $S$ as follows: for $X + \xi \in \Gamma(E)$ and $\omega \in \Gamma(S)$ the Clifford action is given by
\begin{equation*}
(X+\xi) \omega = i_X \omega + \xi \wedge \omega.
\end{equation*}
A non-vanishing section $\rho$ of $S_{\mathbb{C}} = S \otimes \mathbb{C}$ is called a {\em pure spinor} for $J$ if $L$ is the annihilator of $\rho$, where the annihilator of a form $\rho$ is the set of $a \in E$ such that $a\rho = 0$. Every generalized almost complex structure has a pure spinor locally, but not necessarily globally. However such a local pure spinor is unique up to scale, so there exists a complex line bundle $K \subset S$ such that the non-vanishing sections of $K$ are precisely the pure spinors for $J$. We call $K$ the {\em canonical bundle} of $J$. Pure spinors have definite parity, so there exists a $\tau \in \mathbb{Z}_2$ which we call the {\em parity} of $J$ such that $K \subset S^\tau$.\\

Since the $i$ eigenspace $L$ is isotropic we have that the exterior algebra $\wedge^* L^*$ is naturally a subalgebra of the Clifford algebra of $E$. Applying the Clifford action of $L^*$ to $K$, we get an isomorphism $S_{\mathbb{C}} \simeq \wedge^* L^* \otimes K$. This correspondence induces a $\mathbb{Z}$-grading on $S$ of the form:
\begin{equation*}
S_{\mathbb{C}} = U_{-n} \oplus U_{-n+1} \oplus \dots \oplus U_{n-1} \oplus U_n,
\end{equation*}
where $M$ is $2n$-dimensional and $U_k$ is the image of the Clifford action of $\wedge^{k+n} L^*$ on $K$. Additionally we have $\overline{U_k} = U_{-k}$ and the Mukai pairing $\langle \, , \, \rangle$ restricted to $U_j \otimes U_k$ is zero unless $j+k=0$, in which case it is non-degenerate.

We need also an alternative characterization of the induced grading on $S$. Since $J^2 = -1$  and $J$ preserves $\langle \, , \, \rangle$, we find that $\langle Ja,b\rangle + \langle a , Jb \rangle = 0$. Therefore the bilinear form $\langle J \, , \, \rangle$ defines a section of $\wedge^2 E^* \simeq \wedge^2 E$ using $\langle \, , \, \rangle$ to identify $E$ and $E^*$. There is a natural map from $\wedge^2 E$ to the Clifford algebra of $E$ which sends $a \wedge b$ with $a,b \in E$ to $[a,b]/2 = (ab-ba)/2$ in the Clifford algebra. Thus $J$ has a natural Clifford action on $S$. One can show that $U_k$ is the $-ik$-eigenspace of $J$ under this action.\\

We define operators $\partial : \Gamma(U_k) \to \Gamma(U_{k-1})$, $\overline{\partial} : \Gamma(U_k) \to \Gamma(U_{k+1})$ as follows. If $\omega$ is a section of $U_k$ then we may decompose $d_H \omega$ into components of each degree. By definition $\partial \omega$ is the projection of $d_H \omega$ to $U_{k-1}$ and $\overline{\partial} \omega$ is the projection of $d_H \omega$ to $U_{k+1}$. As the notation suggests, $\partial$ and $\overline{\partial}$ are conjugate operators.  For an arbitrary generalized almost complex structure $d_H$ will generally have components of degrees other than $\pm 1$. In fact the identity $d_H = \partial + \overline{\partial}$ is equivalent to integrability of $J$ \cite{gual}. Now if we assume $J$ is integrable then since $d_H^2 = 0$ we immediately obtain the identities $\partial^2 = \overline{\partial}^2 = 0$, $\partial \overline{\partial} + \overline{\partial} \partial = 0$.


\subsection{Smooth and holomorphic families}\label{smhofam}

We would like to have a notion of a family of generalized complex structures which is compatible with our notion of a family of exact Courant algebroids. It turns out that there are two natural ways to do this which correspond roughly to smoothly varying and holomorphically varying families.

\begin{definition}\label{smoothfam}
A {\em smooth family of generalized complex structures} is a fiber bundle $\pi : M \to B$, a family $E \to M$ of exact Courant algebroids over $B$ and a generalized complex structure $J$ on $E$.
\end{definition}
We have seen that given a family of exact Courant algebroids $E \to M$, the restriction $E|_{M_b}$ of $E$ to a fiber $M_b = \pi^{-1}(b)$ naturally inherits an exact Courant algebroid structure. It is straightforward to see that if $J$ is a generalized complex structure on $E$ then $J$ determines by restriction a generalized complex structure on each fiber.\\

Let $M_0,B$ be smooth manifolds and $H \in H^3(M_0,\mathbb{R})$. Let $M = M_0 \times B$, $\pi : M \to B$ the projection to $B$ and $V$ the vertical bundle, which can be identified with the tangent bundle of $M_0$ pulled back to $X$. There is a natural pullback map $H^3(M_0,\mathbb{R}) \to H^3(V,\mathbb{R})$ and so we get a corresponding family of exact Courant algebroids $E\to M$ which has \v{S}evera class $H$ on each fiber. As a bundle we can just take $E$ to be the pullback of $TM_0 \oplus T^*M_0$ to $M$. It is not hard to see that a generalized complex structure on $E$ is precisely a bundle endomorphism $J : E \to E$ such that the restriction of $J$ to each fiber is a generalized complex structure. In other words $J$ is a family of generalized complex structures on $M_0$, smoothly parametrized by $B$. Definition \ref{smoothfam} is more general than this since a class in $H^3(V,\mathbb{R})$ corresponds to a smooth $H^3(M_0,\mathbb{R})$-valued function $H : B \to H^3(M_0,\mathbb{R})$ on $B$. This would correspond to simultaneously deforming the complex structure and \v{S}evera class. Locally (with respect to the base) every smooth family of generalized complex structures has this form.\\

Next we introduce a more rigid notion of a family of generalized complex structure. Before getting to the definition let us explain the general idea. Let $\pi : M \to B$ be a fiber
bundle. This time we want to consider a family of exact Courant algebroids on $M$ that comes from the fiberwise reduction of an exact Courant algebroid $F \to M$ on $M$. Let $V$ be the vertical tangent bundle of $\pi : M \to B$ and define $A,A^\perp$ as in Section \ref{familiesexact}. The corresponding family of exact Courant algebroids is $E = A/A^\perp$. One way to get a complex structure on $A/A^\perp$ is to start with a complex structure on $F$ such that $A$ and $A^\perp$ are complex subbundles. If the complex structure on $F$ respects the pairing $\langle \, , \rangle$ on $F$ and $A$ is a complex subbundle then $A^\perp$ is automatically a complex subbundle as well. Notice that we also get an induced complex structure on $F/A\simeq \pi^*(TB)$. It is natural to demand that this complex structure comes from a complex structure on $B$ (so that $\pi : M \to B$ is in some sense holomorphic). Based on these considerations we make the following definition:
\begin{definition}\label{holofam}
A {\em holomorphic family of generalized complex structures} is a fiber bundle $\pi : M\to B$, an exact Courant algebroid $F \to M$, a generalized complex structure $J$ on $F$ and a complex structure $J_B$ on $B$ such that:
\begin{itemize}
\item{the subbundle $A = {\rm Ker}( \pi_* \circ \rho : F \to \pi^*(TB))$ is $J$-invariant,}
\item{the induced complex structure on the bundle $F/A \simeq \pi^*(TB)$ coincides with $-\pi^*(J_B)$.}
\end{itemize}
\end{definition}
Note the appearance in the above definition of the complex structure $-\pi^*(J_B)$ rather than $\pi_*(J_B)$. This is purely a matter of convention and our choice is motivated by existing conventions in generalized complex geometry.\\

Let $\pi : M \to B$, $F \to M$, $J,J_B$ be a holomorphic family of generalized complex structures. As we have already remarked, since $J$ respects the pairing $\langle \, , \, \rangle$ on $F$ (by definition of a generalized complex structure), we have automatically that $A^\perp$ is also $J$-invariant, so we get an induced generalized almost complex structure $J'$ on $E = A/A^\perp$. In fact since $J$ is integrable it is easy to see that $J'$ is integrable as well. Immediately this implies that $(E , J')$ is a smooth family of generalized complex structures and therefore induces on each fiber a generalized complex structure.\\

Let $M_0$ be a smooth manifold and $(B,J_B)$ a complex manifold. Set $M = M_0 \times B$ and $\pi = \pi_B : M \to B$ the projection to $B$ and similarly let $\pi_{M_0}$ be the projection to $M_0$. Let $H \in H^3(M_0,\mathbb{R})$ and lift $H$ to a class on $M$. Then we get a corresponding exact Courant algebroid $F \to M$ on $M$, namely $F = TM \oplus T^*M$ with the $H$-twisted Dorfman bracket. Since $M$ is the product $M = M_0 \times B$ we get an identification $TM = \pi^*_B(TB) \oplus \pi_{M_0}^*(TM_0)$ and similarly $T^*M = \pi_{M_0}^*(T^*M_0) \oplus \pi_B^*(T^*B)$. Let $E_{M_0}$ be the bundle over $M_0$ given by $E_{M_0} = TM_0 \oplus T^*M_0$. We can give $E_{M_0}$ the structure of an exact Courant algebroid on $M_0$ with $H$-twisted Dorfman bracket. Now write $F =\pi_B^*(TB) \oplus \pi_{M_0}^*( E_{M_0}) \oplus \pi_B^*(T^*B)$. Let $J_{M_0} : \pi^*_{M_0}(E_{M_0}) \to \pi_{M_0}^*(E_{M_0})$ be a bundle endomorphism of $\pi_{M_0}^*(E_{M_0})$ such that $J_{M_0}^2 = -1$ and $J_{M_0}$ preserves the natural pairing on $\pi_{M_0}^*(E_{M_0})$. We may then define a generalized almost complex structure $J$ on $F$ by
\begin{equation}\label{defj}
J = \left[ \begin{matrix} -\pi^*(J_B) & 0 & 0 \\ 0 & J_{M_0} & 0 \\ 0 & 0 & \pi^*(J_B^t) \end{matrix} \right],
\end{equation}
where $J_B^t : T^*B \to T^*B$ is the transpose of $J_B$. Now by a straightforward verification we deduce the following:
\begin{proposition}\label{holofamp}
The generalized almost complex structure $J$ in Equation (\ref{defj}) is integrable if and only if the following conditions hold:
\begin{itemize}
\item{the restriction of $J_{M_0}$ to each fiber of $\pi$ is integrable as a generalized almost complex structure on $E_{M_0}$,}
\item{$J_{M_0}$ depends holomorphically on $B$.}
\end{itemize}
\end{proposition}
Proposition \ref{holofamp} shows that our definition (\ref{holofam}) is able to capture the simple notion of a family of generalized complex structures on $M_0$ which vary holomorphically.


\subsection{Generalized Kodaira-Spencer classes}\label{gksc}

We recall that infinitesimal deformations of generalized complex structure are classified by a class in Lie algebroid cohomology which generalizes the Kodaira-Spencer class for families of complex manifolds. We will then show an alternative interpretation for this generalized Kodaira-Spencer class that emerges naturally from the point of view of smooth or holomorphic families.\\

Let $(M,H,J)$ be a generalized complex manifold. As usual, let $L$ be the $+i$-eigenbundle of the complexified generalized tangent bundle $E_{\mathbb{C}} = (TM \oplus T^*M) \otimes \mathbb{C}$, so $E_{\mathbb{C}} = L \oplus \overline{L}$. Suppose $J'$ is another generalized almost complex structure and let $L'$ be the $+i$-eigenbundle of $J'$. For any such $J'$ sufficiently close to $J$ pointwise, we have that $L'$ can be written as a graph of a map $\epsilon : L \to \overline{L}$:
\begin{equation}\label{ieigdef}
L' = \{ X + \epsilon X \, | \, X \in L \}.
\end{equation}
Recall that $L,\overline{L}$ are isotropic subspaces of $E$ and that the pairing on $E$ identifies $\overline{L}$ with $L^*$. We can therefore think of $\epsilon$ as a section of $L^* \otimes L^*$. We find that $L'$ is isotropic if and only if $\epsilon \in \Gamma(M,\wedge^2 L^*)$. Conversely all sufficiently small such $\epsilon$ determine a generalized almost complex structure by (\ref{ieigdef}). The generalized almost complex structure $J'$ corresponding to $\epsilon$ is integrable if and only if $\epsilon$ satisfies a Maurer-Cartan equation \cite{gual}
\begin{equation}\label{mcgencpx}
d_L \epsilon + \frac{1}{2} [\epsilon , \epsilon] = 0,
\end{equation}
where $d_L$ is the Lie algebroid differential and $[ \, , \, ]$ is the Schouten bracket $[ \, , ] : \Gamma(M,\wedge^j L^*) \otimes \Gamma(M,\wedge^k L^*) \to \Gamma(M,\wedge^{j+k-1} L^*)$ induced by the Lie algebroid structure on $L^*$. Now if we vary $J$ in a smooth family $\{ J(b) \}_{b \in B}$ with $J(0) = J$, we get (for sufficiently small deformations) a corresponding smooth family $\epsilon(b)$ of solutions to the Maurer-Cartan equation, with $\epsilon(0) = 0$. Given a tangent vector $T \in T_0B$, let $\epsilon_T = T(\epsilon)$ be the derivative of $\epsilon$ at $0$ in the direction $T$. We find that equation (\ref{mcgencpx}) linearizes to $d_L ( \epsilon_T ) = 0$, so $\epsilon_T$ defines a class $[\epsilon_T] \in H^2(L)$ in Lie algebroid cohomology. This gives a map $\rho : T_0B \to H^2(L)$, which we call the {\em generalized Kodaira-Spencer class}. When $M$ is compact it is possible to show that automorphisms of the $H$-twisted Courant bracket on $TM \oplus T^*M$ generated by the adjoint action of vector fields and $1$-forms give rise to deformations of $J$ with trivial generalized Kodaira-Spencer class. The generalized Kodaira-Spencer class is therefore a measure of non-trivial deformations.\\  

We now consider an interpretation of generalized Kodaira-Spencer classes from the point of view of smooth families of generalized complex families. Let $\pi : M \to B$ be a fiber bundle with fiber $M_0$ and $E \to M$ a smooth family of exact Courant algebroids on $M$. We will also assume that $E$ is the fiberwise reduction of an exact Courant algebroid $F \to M$ on $M$. The idea is that vector fields on $M$ can in some sense be used to generate non-trivial derivations of the Courant algebroid $E$. First, let us define what we mean by a derivation.
\begin{definition}
Let $(E,[ \, , \, ], \langle \, , \, \rangle , \rho)$ be a Courant algebroid on a manifold $M$. A {\em derivation} of $E$ is a linear map $D : \Gamma(E) \to \Gamma(E)$ on the sections of $E$ such that there exists a vector field $V$ on $M$ so that the pair $(D,V)$ satisfies the following identities:
\begin{itemize}
\item[(D1)]{$D(fx) = V(f)x + fDx$,}
\item[(D2)]{$D[x,y] = [Dx,y] + [x,Dy]$,}
\item[(D3)]{$V\langle x , y \rangle = \langle Dx , y \rangle + \langle x , Dy \rangle$,}
\end{itemize}
for all $x,y \in \Gamma(E)$ and functions $f$.
\end{definition}
\begin{remark}
The vector field $V$ is uniquely determined by the derivation $D$. Axiom (D1) states that $D$ is a covariant differential operator \cite{mac}. Also from the definition it follows that for all $x \in \Gamma(E)$ we have $\rho(Dx) = [V , \rho(x)]$.
\end{remark}
We observe that for any Courant algebroid $E$ and any section $x$ of $E$ we get an induced derivation $ad_x : \Gamma(E) \to \Gamma(E)$ defined by $ad_x(y) = [x,y]$. Such a derivation may be called an inner derivation. More interesting are derivations which are not inner. We now show how a Courant algebroid which is a fiberwise reduction admits many derivations that are not inner. Let $\pi : M \to B$ be a locally trivial fiber bundle, $F$ an exact Courant algebroid on $M$ and $E$ the fiberwise reduction of $F$. For definiteness let $F = TM \oplus T^*M$ with the $H$-twisted Dorfman bracket, for a closed $3$-form $H \in \Omega^3(M)$. Let $Z$ be a vector field on $B$. Choose a vector field $\tilde{Z}$ on $M$ such that for all $x \in M$ we have $\pi_*( \tilde{Z}(x) ) = Z(\pi(x))$. We claim that the inner derivation $ad_{\tilde{Z}}$ of $F$ induces a derivation on $E$ which is not inner, except when $Z = 0$.

Let $A$ be the subbundle of $F$ which is the kernel of $\pi_* \circ \rho : F \to \pi^*(TB)$. It is clear that $ad_{\tilde{Z}}$ sends $\Gamma(A)$ to $\Gamma(A)$ since for any vertical vector field $V$ we have that $[\tilde{Z},V]$ is also vertical. Let $A^\perp$ be as usual the annihilator of $A$ in $F$. We find that $ad_{\tilde{Z}}$ sends $\Gamma(A^\perp)$ to $\Gamma(A^\perp)$. In fact this follows from the fact that $ad_{\tilde{Z}}$ preserves $\Gamma(A)$ and property (D3) for derivations. Thus $ad_{\tilde{Z}}$ induces a map $D : \Gamma(A/A^\perp) \to \Gamma(A/A^\perp)$. Now recall that $E = A/A^\perp$ is Courant algebroid, the fiberwise reduction of $F$. It is straightforward to see that $D$ is a derivation of $E$. Moreover $D$ can not be inner unless $Z = 0$. We note that $D$ depends of both the choice of lift $\tilde{Z}$ of $Z$ and the choice of isotropic splitting for $T^*M \to F \to TM$. However for a given vector field $Z$ on the base $B$, one sees that the different choices give rise to derivations of $E$ which differ by an inner derivation. Any derivation $D$ associated to a vector field $Z$ on $B$ in this manner will be called a lift of $Z$ to a derivation of $E$.\\

Let $J$ be a generalized complex structure on $E$. Let $L$ be the $+i$-eigenspace of $J$ on $E_{\mathbb{C}} = E \otimes \mathbb{C}$ so that $E_{\mathbb{C}} = L \oplus \overline{L}$. Let $Z$ be a vector field on $B$ and $D$ a lift of $Z$ to a derivation of $E$ as previously described. Consider now the following map $\epsilon : \Gamma(L) \times \Gamma(L) \to \mathcal{C}^\infty(M)$ given as follows:
\begin{equation}
\epsilon(a,b) = \langle Da , b \rangle,
\end{equation}
where $a,b$ are sections of $L$. Since $L$ is isotropic it follows that $\epsilon$ is $\mathcal{C}^\infty(M)$-linear in $a,b$. Moreover $\epsilon$ is skew-symmetric by (D3), so we may view $\epsilon$ as a section of $\wedge^2 L^*$.
\begin{proposition}
The class $\epsilon \in \wedge^2 L^*$ is closed under the Lie algebroid differential. Moreover the cohomology class $[\epsilon] \in H^2(L)$ depends only on the vector field $Z$, not the choice of derivation $D$ lifting $Z$.
\end{proposition}
\begin{proof}
Let $a,b,c \in \Gamma(L)$. Then
\begin{align*}
(d_L \epsilon)(a,b,c) =& \rho(a)\langle Db , c \rangle - \rho(b)\langle Da , c \rangle + \rho(c)\langle Da,b \rangle \\
& - \langle D[a,b],c \rangle + \langle D[a,c] , b \rangle - \langle D[b,c] , a \rangle \\
=& \rho(a)\langle Db,c \rangle + \rho(b)\langle Dc , a \rangle + \rho(c)\langle Da , b \rangle \\
&+ \langle Dc , [a,b] \rangle + \langle Db , [c,a] \rangle + \langle Da , [b,c] \rangle \\
=& \langle [a,Db] , c \rangle + \langle [b,Dc],a \rangle + \rho(c)\langle Da , b \rangle + \langle Da, [b,c] \rangle.
\end{align*}
Now since $[a,Db] + [Db,a] = 2\langle a , Db \rangle$, we find that $\langle [a,Db] , c \rangle + \rho(c)\langle a,Db \rangle = - \langle [Db,a] , c \rangle$. Substituting we obtain
\begin{align*}
(d_L \epsilon)(a,b,c) &= -\langle [Db,a],c\rangle + \langle [b,Dc],a\rangle + \langle Da,[b,c] \rangle \\
&= \langle a , [Db,c] \rangle + \langle a , [b,Dc] \rangle + \langle Da,[b,c] \rangle \\
&= \langle a , D[b,c] \rangle + \langle Da,[b,c]\rangle \\
&= 0,
\end{align*}
so $d_L \epsilon = 0$ as claimed.\\

To show that the cohomology class of $\epsilon$ depends only on $Z$ recall that any two derivations of $E$ associated to $Z$ differ by an inner derivation. Thus it suffices to show that for any section $x$ of $E$ the element $\delta \in \Gamma(\wedge^2 L^*)$ given by $\delta(a,b) = \langle [x,a],b \rangle$ is exact. Write $x = c + \overline{c}$ where $c \in \Gamma(L)$. We can think of $\overline{c}$ as a section of $L^* \simeq \overline{L}$. We then find
\begin{align*}
(d_L \overline{c})(a,b) &= \rho(a)\langle \overline c , b \rangle - \rho(b)\langle \overline{c} , a \rangle - \langle \overline{c} , [a,b] \rangle \\
&= \langle [a,\overline{c}],b\rangle + \langle \overline{c} , [a,b] \rangle - \langle b , [a,\overline{c}] + [\overline{c},a] \rangle - \langle \overline{c} , [a,b] \rangle \\
&= -\langle b , [\overline{c},a] \rangle \\
&= - \langle [c + \overline{c} , a] , b \rangle \\
&= -\delta(a,b),
\end{align*}
which shows that $\delta$ is exact.
\end{proof}
We have shown that there exists a natural map $\kappa : \Gamma(TB) \to H^2(L)$ which sends a vector field $Z$ on $B$ to the corresponding Lie algebroid cohomology class. We would like to interpret $\kappa$ as describing the infinitesimal deformations of a family of generalized complex structures. Note that if $U \subset B$ is an open subset of $B$ then we can restrict the Courant algebroid $E$ and generalized complex structure on $E$ to the open set $\pi^{-1}(U)$ of $M$. In this way $H^2(L)$ becomes a presheaf on $B$ and the map $\kappa$ is a presheaf map. In Section \ref{familiesexact} we explained how it is possible to restrict the Courant algebroid $E$ to any fiber of $M$. The same reasoning applies to $L$ and also applies at the level of Lie algebroid cohomology as we now explain. Let $b \in B$ and let $M_b = \pi^{-1}(b)$ be the fiber over $b$. We let $L_b$ denote the Lie algebroid obtained by restricting $L$ to $M_b$. Thus $L_b = L|_{M_b}$ as a vector bundle and the Lie algebroid bracket is defined by $[a,b]_{L_b} = ([\tilde{a},\tilde{b}]_L)|_{M_b}$, where $\tilde{a},\tilde{b}$ are smooth extensions of $a,b$ to sections of $L$. In a similar manner one sees that the Lie agebroid differential for $L,L_b$ are related as follows:
\begin{equation*}
(d_{L_b} \omega)(a_0,\dots,a_{k+1}) = (d_L \tilde{\omega})(\tilde{a_0}, \dots , \tilde{a_{k+1}})|_{M_b},
\end{equation*}
where $\omega \in \Gamma(\wedge^k L_b^*)$, $a_0,\dots , a_{k+1} \in \Gamma(L_b)$ and $\tilde{\omega},\tilde{a_0},\dots,\tilde{a_{k+1}}$ are smooth extensions. It follows that for any $b \in B$ there is a natural restriction map $H^2(L) \to H^2(L_b)$.
\begin{proposition}
Let $Z$ be a vector field on $B$. For any $b \in B$ let $\kappa_b(Z)$ denote the image of $\kappa(Z)$ in $H^2(L_b)$. Then $\kappa_b(Z)$ only depends on the value of $Z$ at $b$. Moreover, if we choose a local trivialization $\pi^{-1}(U) = U \times M_b$ for a neighborhood $U$ of $b$ so that the generalized complex structure on $E$ can be viewed as a generalized complex structure on $M_b$ that varies smoothly with $U$, then $\kappa_b(Z)$ is the generalized Kodaira-Spencer class for the deformation of generalized complex structure on $M_b$ in the direction $Z_b \in T_bU$.
\end{proposition}
\begin{proof}
Choose a local trivialization $\pi^{-1}(U) = U \times M_b$ as in the statement of the proposition. By restriction we may assume $U$ is contractible so that $H^3(U \times M_b , \mathbb{R}) = H^3(M_b,\mathbb{R})$. It follows that there exists a closed $3$-form $H$ on $M_b$ so that the restriction of $F$ to $U \times M_b$ is isomorphic to $TU \oplus TM_b \oplus T^*M_b \oplus T^*U$ with $H$-twisted Courant bracket. The Courant algebroid $E$ is obtained from $F$ by first restricting to sections of $TM_b \oplus T^*M_b \oplus T^*U$ and then factoring out by sections of $T^*U$. In particular $E$ is isomorphic to the bundle $TM_b \oplus T^*M_b$ pulled back to $U \times M_b$. Let $Z$ be a vector field on $M$ which is a vector field on $U$ by restriction. Then a derivation $D : \Gamma(E) \to \Gamma(E)$ of $E$ associated to $Z$ is given by
\begin{equation*}
D( X + \xi ) = \mathcal{L}_Z( X + \xi),
\end{equation*}
where $X + \xi$ is a section of $E$ thought of as a vector field and $1$-form on $U \times M_b$. Let $a,b$ be sections of $L_b$. Choose smooth extensions $\tilde{a},\tilde{b}$ of $a,b$ to sections of $L$ over $U \times M_b$. Then we find that $\kappa_b(Z)$ is represented by $\epsilon(a,b)$ where
\begin{equation*}
\epsilon(a,b) = \langle \mathcal{L}_Z (\tilde{a}) , \tilde{b} \rangle |_{M_b}.
\end{equation*}
For each $u \in U$ we can think of $L_u$ as a subbundle of $TM_b \oplus T^*M_b$. By suitably restricting $U$ we can then write each $L_u$ as a graph over $L_b$ for some $\lambda_u \in \Gamma( \wedge^2 L_b^*)$, that is $L_u = \{ X + \lambda_u X \, | \, X \in L_b \}$. In particular this allows us to choose the extensions $\tilde{a},\tilde{b}$ as follows $\tilde{a}(u,m) = a(m) + \lambda_u a(m)$, $\tilde{b}(u,m) = b(m) + \lambda_u b(m)$, where $u \in U$ and $m \in M_b$. Noting that $\lambda_b = 0$ we arrive at
\begin{equation*}
\epsilon(a,b)(m) = \langle \mathcal{L}_Z(\lambda_u)|_{u=b} \, a(m) , b(m) \rangle.
\end{equation*}
Note that the cohomology class of $\mathcal{L}_Z \lambda |_{u=b}$ in $H^2(L_b)$ is precisely the generalized Kodaira-Spencer class for the deformation of generalized complex structure in the direction $Z_b$ which clearly only depends on the value of $Z$ at $b$.
\end{proof}

So far we have assumed that $E$ is the fiberwise reduction of an exact Courant algebroid $F$ on $X$. Now we will assume that $(E,J)$ is a holomorphic family of generalized complex structures. Thus there is a generalized complex structure $J'$ on $F$ which induces $J$ on $E$, as in Definition \ref{holofam}.
\begin{proposition}\label{ksmc}
For a holomorphic family of generalized complex structures the Kodaira-Spencer map $\kappa : TB \to H^2(L)$ is complex linear in the sense that
\begin{equation}\label{kappaholo}
\kappa(IX) = i\kappa(X),
\end{equation}
where $I$ is the complex structure on $B$ and $X$ is a vector field on $B$.
\end{proposition}
\begin{proof}
We may assume $F = TM \oplus T^*M$ with $H$-twisted Courant bracket for some closed $3$-form $H$ on $X$. Let $L' \subset F$ be the $+i$-eigenspace of $J'$. Let $X$ be a vector field on $B$, $\tilde{X}$ a horizontal lift to $X$. Then under the given choice of splitting for $F$ we can think of $\tilde{X}$ as a section of $F$ and thus we get a derivation of $E$ given by $De = [\tilde{X},e]$. We then have that $\kappa(X)$ is represented by the element of $\Gamma(\wedge^2 L^*)$ given by $a,b \mapsto \langle Da,b \rangle$. Now observe that since we have a holomorphic family it follows that $-J' \tilde{X}$ is a section of $F$ which lifts $IX$. Thus $\kappa(IX)$ is represented by $\langle -[J'\tilde{X} , a] , b \rangle$. Let $\tilde{X} = \tilde{X}^{1,0} + \tilde{X}^{0,1}$, where $\tilde{X}^{1,0}$ is a section of $L'$ and $\tilde{X}^{0,1}$ is a section of $\overline{L'}$. Then $-J'\tilde{X} = -i\tilde{X}^{1,0} + i\tilde{X}^{0,1}$ and it follows that $\langle -[J'\tilde{X} ,a],b \rangle = i \langle [\tilde{X},a],b\rangle$ since $\langle [\tilde{X}^{1,0},a],b\rangle = 0$ using integrability of $J'$. The identity (\ref{kappaholo}) immediately follows.
\end{proof}


\section{Variation of Hodge structure}\label{varhodstr}

\subsection{Hodge structures on generalized complex manifolds}\label{hsgcm}

Let $E \to M$ be an exact Courant algebroid on $M$, where $M$ is $2n$-dimensional and $J$ a generalized complex structure on $E$. Suppose $E$ has \v{S}evera class represented by a closed $3$-form $H$, so we can identify $E$ with $TM \oplus T^*M$ equipped with the $H$-twisted Dorfman bracket. The generalized complex structure $J$ induces a decomposition of forms $S = U_{-n} \oplus U_{-n+1} \oplus \cdots \oplus U_{n-1} \oplus U_n$ such that the twisted differential decomposes as $d_H = \partial + \overline{\partial}$, where $\partial$ has degree $-1$ and $\overline{\partial}$ has degree $+1$. Since $\overline{\partial}^2 = 0$, we have corresponding cohomology groups, denoted $H^k_{\overline{\partial}}(M)$.\\

Since $S$ only has a single grading rather than a bi-grading it may seem at first there is no spectral sequence relating cohomology of $\overline{\partial}$ and $d_H$. There is a simple way to get around this fact however. We set $W^{p,q} = \Gamma( U_{p-q})$. Let $\delta : W^{p,q} \to W^{p+1,q}$ be given by $\overline{\partial}$ and let $\delta' : W^{p,q} \to W^{p,q+1}$ be given by $\partial$. Then it is clear that $(W^{*,*},\delta,\delta')$ defines a bi-graded complex. Consider the associated singly graded complex $(W^* , D)$, where $W^k = \oplus_{p+q = k} W^{*,*}$ and $D = \delta + \delta'$. One finds that $W^k = \oplus_{r = k \,( {\rm mod} \,2)} U_r = S^{k+n+\tau}$ and $D = d_H$. The cohomology $H^k(W^*)$ is therefore the twisted cohomology group $H^{k+n+\tau}(M,H) \otimes \mathbb{C}$. The shift in degrees is due to the fact that $U_{-n} = K$ has parity $\tau$. If we filter $W^{p,q}$ by $p$-degree we get a spectral sequence $E_r^{p,q}$ converging to the twisted cohomology of $(M,H)$ and such that $E_1^{p,q} = H_{\overline{\partial}}^{p-q}(M)$. We call this the generalized Fr\"olicher spectral sequence. Alternatively we may filter by $q$-degree to get a second spectral sequence ${E'}_r^{p,q}$ such that (after swapping $p$ and $q$ indices) ${E'}_1^{p,q} = H^{q-p}_{\partial}(M)$. Note however that this spectral sequence (and corresponding filtration on twisted cohomology) is just the complex conjugate of the first.

Associated to the spectral sequence $E_r^{p,q}$ is a decreasing filtration $\hat{F}^p H^k(M,H)$ on twisted cohomology such that $\hat{F}^p H^k(M,H) / \hat{F}^{p+1} H^k(M,H) \simeq E_\infty^{p,q}$. Note that by construction of the spectral sequence $\hat{F}^{a+1} H^{b+2}(M,H) = \hat{F}^a H^b(M,H)$. We will define a modified filtration $F^p H(M)$ by setting
\begin{equation*}
F^{k-2p} H(M) = \overline{\hat{F}^p H^k(M,H)}.
\end{equation*}
The change of indices and conjugation are for later convenience. Note also that we only need a single index, since from the above definition we know $F^pH(M) \subseteq H^{p+n+\tau}(M,H)_\mathbb{C}$. One should think of $F^p H(M)$ as elements of twisted cohomology of parity $p+n+\tau$ mod $2$ and represented as a sum of terms in $U_j$ for $j \le p$. The filtrations are now increasing and take the form:
\begin{align*}
F^{-n}H(M) \subseteq F^{-n+2}H(M) \subseteq &\dots \subseteq F^{n} H(M) = H^{\tau}(M,H)_\mathbb{C}, \\
F^{-n+1}H(M) \subseteq F^{-n+3}H(M) \subseteq &\dots \subseteq F^{n-1} H(M) = H^{\tau+1}(M,H)_\mathbb{C}.
\end{align*}
We are interested in the case that the twisted cohomology can be decomposed into the $\overline{\partial}$-cohomology. For this to happen it is necessary first of all for the spectral sequence to
degenerate at $E_1$, thus $\hat{F}^p H(M) / \hat{F}^{p-2} H(M) = H^p_{\overline{\partial}}(M)$. However this is still not sufficient. We also need that the filtration $F^p H(M)$ induces a Hodge structure on $H^{p+n+\tau}(M,H)$, namely the natural maps $F^pH(M) \oplus \overline{ F^{-p-2} H(M)} \to H^{p+n+\tau}(M,H)_\mathbb{C}$ need to be isomorphisms. If both of these conditions hold then we get induced decompositions
\begin{equation*}
H^{k+n+\tau}(M,H) \otimes \mathbb{C} = \bigoplus_{p = k \, ({\rm mod} \, 2)} H^p_{\overline{\partial}}(M).
\end{equation*}
Thus $H^{\tau}(M,H)$ will have a Hodge structure of weight $n$ and $H^{\tau+1}(M,H)$ a Hodge structure of weight $(n-1)$.\\

Let $M,J$ be a generalized complex manifold. We recall the definition of the $\partial \overline{\partial}$-lemma, also known as the $dd^J$-lemma:
\begin{definition}
We say that $M,J$ satisfies the $\partial \overline{\partial}$-lemma if
\begin{equation*}
Im(\partial) \cap Ker(\overline{\partial}) = Im(\overline{\partial}) \cap Ker(\partial) = Im(\partial \overline{\partial}).
\end{equation*}
\end{definition}
Following the argument used in \cite{dgms}, one finds that the $\partial \overline{\partial}$-lemma is equivalent to the following two conditions:
\begin{enumerate}
\item{the generalized Fr\"olicher spectral sequence degenerates at $E_1$, and}
\item{the induced filtration on twisted cohomology is a Hodge filtration.}
\end{enumerate}
Actually in \cite{dgms} they work with bounded complexes. The corresponding complex here $W^{r,s}$ is not bounded, but it is the case $W^{r,s} = 0$ if $|r-s| > n$, which is sufficient for the argument of \cite{dgms} to carry through.\\

The $\partial \overline{\partial}$-lemma implies the existence of a Hodge decomposition
\begin{equation*}
H^{k+n+\tau}(M,H)_\mathbb{C} = \bigoplus_{p = k \, ({\rm mod} \, 2)} H^p_{\overline{\partial}}(M)
\end{equation*}
in twisted cohomology. In order to consider variations of this Hodge structure we need to know that under suitable conditions the existence of the Hodge decomposition is stable under sufficiently small deformations. In fact, the usual elliptic semi-continuity argument used in the complex setting can be applied to obtain:
\begin{proposition}\label{ddbstab}
Let $\pi : M \to B$ be a locally trivial fiber bundle with compact fibers. Let $E \to M$ be a family of exact Courant algebroids on $M$ obtained by reduction of an exact Courant algebroid $F$ on $M$ and $J$ a generalized complex structure on $E$. Suppose that for some $b \in B$ the generalized complex structure on the fiber $M_b = \pi^{-1}(b)$ satisfies the $\partial \overline{\partial}$-lemma. Then there exists an open neighborhood $U \subseteq B$ of $b$ in $B$ such that for every $u \in U$ the generalized complex structure on the fiber $M_u$ also satisfies the $\partial \overline{\partial}$-lemma.
\end{proposition}
\begin{proof}
Since $E$ arises from reduction of an exact Courant algebroid on $M$, we know by Section \ref{twgmconn} that the twisted cohomology groups $H^*(M_u , H|_{M_u})$ of the fibers have constant dimension and in fact define a flat vector bundle over $B$. Now since $M_u$ satisfies the $\partial \overline{\partial}$-lemma we have that the sums of the dimensions of the $\overline{\partial}$-cohomology groups $H^*_{\overline{\partial}}(M_b)$ agrees with the dimension of $H^*(M_b,H_{M_b})$. Now the $\overline{\partial}$-operator forms an elliptic complex and the $H^*_{\overline{\partial}}(M_u)$ groups are then in bijection with the kernels of the corresponding Laplacians. By elliptic semi-continuity \cite[Theorem 7.3]{kod} there is a neighborhood $U$ of $b$ such that ${\rm dim}( H^k_{\overline{\partial}}(M_u)) \le {\rm dim}(H^k_{\overline{\partial}}(M_b))$ for all $u \in U$. On the other hand by considering the generalized Fr\"olicher spectral sequence we find that
\begin{align*}
{\rm dim}\left( H^0(M_u,H|_{M_u}) \oplus H^1(M_u,H|_{M_u}) \right) &\le {\rm dim}\left( \bigoplus_k H^k_{\overline{\partial}}(M_u) \right) \\
&\le {\rm dim}\left( \bigoplus_k H^k_{\overline{\partial}}(M_b) \right) \\
&= {\rm dim}\left( H^0(M_b,H|_{M_u}) \oplus H^1(M_b,H|_{M_u}) \right) \\
&= {\rm dim}\left( H^0(M_u,H|_{M_u}) \oplus H^1(M_u,H|_{M_u}) \right).
\end{align*}
This requires equalities throughout so for all $u \in U$ we see that the generalized Fr\"olicher spectral sequence degenerates at $E_1$ and ${\rm dim}(H^k_{\overline{\partial}}(M_u)) = {\rm dim}(H^k_{\overline{\partial}}(M_b))$. In addition, since $F^p H(M_b) \cap \overline{F^{-p-2}H(M_b)} = 0$, the same is true for all $u$ sufficiently close to $b$, since we can view twisted cohomology as defining a smooth vector bundle over $U$ and the $F^p H(M_u)$ are smoothly varying subbundles. From this it immediately follows that for all $u$ sufficiently close to $b$ the conditions equivalent to the $\partial \overline{\partial}$-lemma hold.
\end{proof}

Before we proceed to consider variations of Hodge structure, let us consider what maps between generalized complex manifolds preserve Hodge structures. We obtain a result analogous to the complex case \cite[Theorem 2.4]{gr1}. First we need to explain the notion of Poincar\'e dual classes in twisted cohomology.

Let $M$ be a compact oriented manifold and $H$ a closed $3$-form on $M$. Suppose we have a pair $(N,B)$ consisting of a submanifold $N$ and $2$-form $B$ on $N$ such that $H|_N = dB$. In this case we get a linear map $\lambda_{N,B} : H^*(M,H) \to \mathbb{R}$ as follows. Let $\omega$ be a $d_H$-closed form on $M$. Then we define $\lambda_{N,B}(\omega) = \int_N e^B \omega$, where the integration denotes an integration over the top degree part of $e^B (\omega|_N)$. One checks that this map evaluates to zero on $d_H$-exact classes and so passes to a linear map on twisted cohomology. The Mukai pairing $\langle \, , \, \rangle$ defines a non-degenerate pairing on twisted cohomology $Q : H^*(M,H) \otimes H^*(M,H) \to \mathbb{R}$ and so there exists a class $\phi_{N,B} \in H^*(M,H)$ which is dual to $\lambda_{N,B}$. Explicitly this means that for any class $\omega \in H^*(M,H)$ we have
\begin{equation*}
\lambda_{N,B}(\omega) = Q( \omega , \phi_{N,B}) = \int_M \langle \omega , \phi_{N,B} \rangle.
\end{equation*}
One may think of $\lambda_{N,B}$ as a sort of twisted fundamental class and $\phi_{N,B}$ as the Poincar\'e dual.\\

Let $(M,J,H),(M',J'.H')$ be generalized complex manifolds satisfying the $\partial \overline{\partial}$-lemma. Let $f : M \to M'$ be a diffeomorphism and $B$ a $2$-form on $M$ such that $H = f^*(H') + dB$. Such a pair $(f,B)$ gives rise to a pullback map $(f,B)^* : H^*(M',H') \to H^*(M,H)$ which sends a $d_{H'}$-closed form $\omega$ to $e^{-B}f^*(\omega)$. We give a characterization of pairs $(f,B)$ which preserve Hodge structures. Let $p,p'$ denote the projections from $M \times M'$ to $M,M'$. We then have an isomorphism $H^*(M,H) \otimes H^*(M',-H') \to H^*(M\times M' , p^*(H)-p'^*(H'))$ which sends $\omega \otimes \omega'$ to $p^*(\omega) \wedge p'^*(\omega')$. To see that this is an isomorphism one only needs to observe that there is a similar tensor product decomposition at the level of differential forms from which the decomposition in cohomology follows. Let $\sigma : H^*(M',H') \to H^*(M',-H')$ be the isomorphism in Section \ref{twgmconn}. We then get a $\mathbb{Z}$-graded decomposition $H^*(M\times M' , p^*(H)-p'^*(H')) = \bigoplus_{k} H^k$ by setting
\begin{equation}\label{productdecomp}
H^k = \bigoplus_{i+j=k} H^i_{\overline{\partial}}(M) \otimes \sigma( H^j_{\overline{\partial}}(M)).
\end{equation}
Let $\sigma'$ denote the involution of $TM' \oplus T^*M'$ which acts as the identity on $TM$ and as $-1$ on $T^*M'$. Then $J'_{\sigma'} = \sigma' J' \sigma'$ is a generalized complex structure on $M'$ and one can show that the decomposition $H^*(M\times M' , p^*(H)-p'^*(H')) = \bigoplus_{k} H^k$ with $H^k$ defined in (\ref{productdecomp}) is the decomposition corresponding to the generalized complex structure $p^*(J) \oplus p'^*(J'_{\sigma'})$ on $M \times M'$. Now let $N = \{ (m,m') \in M \times M' \, | \, m' = f(m) \}$ be the graph of $f$. Let $H_0 = p^*(H) - p'^*(H')$ and $B_0 = p^*(B)|_N$. Then it follows that $H_0|_N = dB_0$, so the pair $(N,B_0)$ defines a Poincar\'e dual class $\phi_{N,B_0} \in H^*(M \times M' , H_0)$.
\begin{proposition}
The map $(f,B)^* : H^*(M',H') \to H^*(M,H)$ preserves Hodge decompositions (in the sense that $(f,B)^*$ sends $H^k_{\overline{\partial}}(M')$ to $H^k_{\overline{\partial}}(M)$) if and only if $\phi_{N,B_0} \in H^0$ where $H^0$ is defined in (\ref{productdecomp}).
\end{proposition}
\begin{proof}
Consider a class in $H^*(M \times M' , H_0)$ of the form $p^*(\alpha) \wedge \sigma(p'^*\beta)$ where $\alpha \in H^*(M,H)$, $\beta \in H^*(M',H')$. We know that such classes span $H^*(M\times M' , H_0)$. By definition of $\phi_{N,B_0}$ we have
\begin{align*}
Q( p^*(\alpha) \wedge \sigma(p'^*\beta) , \phi_{N,B}) &= \int_N e^{B_0} p^*(\alpha) \wedge \sigma( p'^*(\beta)) \\
&= \int_M \alpha \wedge \sigma( e^{-B} f^*(\beta) ) \\
&= \int_M \langle \alpha , (f,B)^* \beta \rangle.
\end{align*}
Now $\phi_{N,B_0}$ has degree $0$ if and only if $Q( p^*(\alpha) \wedge \sigma(p'^*\beta) , \phi_{N,B} ) = 0$ for all $\alpha$ and $\beta$ whose degrees do not sum to zero. One the other hand $\int_M \langle \alpha , (f,B)^* \beta \rangle = 0$ for all $\alpha$ and $\beta$ whose degrees do not sum to zero if and only if $(f,B)^* \beta$ has the same degree as $\beta$, which is exactly the condition for $(f,B)^*$ to preserve Hodge structures.
\end{proof}


\subsection{Variation of Hodge structure}\label{vhs}

Let $\pi : M \to B$ be a locally trivial fiber bundle with compact fibers. Let $E \to M$ be a family of exact Courant algebroids obtained by fiberwise reduction of an exact Courant algebroid $F \to M$ on $M$. Let $J$ be a generalized complex structure on $E$. Assume in addition that the fibers of $\pi$ satisfy the $\partial \overline{\partial}$-lemma with respect to their induced generalized complex structure. We will refer to such a family $(M,B,\pi,E,J)$ as a {\em good family}. Note that if $M_0$ is a compact generalized complex manifold satisfying the $\partial \overline{\partial}$-lemma then by Proposition \ref{ddbstab} we know that all sufficiently small families of deformations of $M_0$ are good families.

We have that $E$ is the fiberwise reduction of an exact Courant algebroid $F$ on $M$. For definiteness we take $F = TM \oplus T^*M$ with $H$-twisted Courant bracket for some closed $3$-form $H \in \Omega^3(M)$. The twisted cohomology groups of the fibers of a good family form a local system $\mathcal{H}^*$, which defines a corresponding flat vector bundle $(H^* , \nabla)$, so $\mathcal{H}^*$ is the sheaf of constant sections of $H^*$. As in Section \ref{twgmconn}, we call $\nabla$ the twisted Gauss-Manin connection. Since the fibers of $\pi$ satisfy the $\partial \overline{\partial}$-lemma we know that on each fiber there is a Hodge structure
\begin{equation*}
H^{k+n+\tau}(M_b)_\mathbb{C} = \bigoplus_{p = k \, ({\rm mod} \, 2)} H^p_{\overline{\partial}}(M_b),
\end{equation*}
where $M_b = \pi^{-1}(b)$, $2n$ is the dimension of the fibers and $\tau$ is the parity of the generalized complex structure on the fibers. As in the proof of Proposition \ref{ddbstab} we know that the dimensions of the $H_{\overline{\partial}}^k(M_b)$ groups are locally constant on $b$. For simplicity we will assume $B$ is connected, so the dimension of each group $H^k_{\overline{\partial}}(M_b)$ is constant. By standard elliptic theory \cite[Theorem 7.4]{kod} we know that the assignment $b \mapsto H^k_{\overline{\partial}}(M_b)$ is a smoothly varying subbundle of $H^*$ which we denote by $H^k_{\overline{\partial}}$. The Hodge decomposition can be written as
\begin{equation*}
H^{k+n+\tau}_\mathbb{C} = \bigoplus_{p = k \, ({\rm mod} \, 2)} H^p_{\overline{\partial}}.
\end{equation*}
We note that $\overline{H^k_{\overline{\partial}}} = H^{-k}_{\overline{\partial}}$. We also note that the Mukai pairing $Q : H^* \otimes H^* \to \mathbb{R}$ upon complexification induces a duality pairing between $H^k_{\overline{\partial}}$ and $H^{-k}_{\overline{\partial}}$. We have a Hodge filtration on each fiber and this defines a corresponding filtration on $H^*$. Namely for $p = k \, ({\rm mod} \, 2)$ we set 
\begin{equation*}
F^p H = \bigoplus_{k \le p, \, \, k = p \, ({\rm mod} \, 2)} H^k_{\overline{\partial}}.
\end{equation*}
On any given fiber $M_b$ this is just the previously defined Hodge filtration $F^pH(M_b)$ for $H^*(M_b,H_b)$. We now show that the analogue of Griffiths transversality \cite{gr2} holds:
\begin{proposition}\label{grtrans}
If $s$ is a section of $F^pH$ and $\nabla$ is the twisted Gauss-Manin connection then $\nabla s$ is a section of $T^*B \otimes F^{p+2}H$. We therefore get an induced bundle map $\tau : F^{p}H / F^{p-2}H \to T^*B \otimes F^{p+2}H / F^{p}H$, that is a bundle map $\tau : TB \otimes H^{p}_{\overline{\partial}} \to H^{p+2}_{\overline{\partial}} $. Let $\kappa : \Gamma(TB) \to H^2(L)$ be the generalized Kodaira-Spencer map. For each $Z \in T_bB$ we have a generalized Kodaira-Spencer class $\kappa_b(Z) \in H^2(L_b)$, where $L_b$ is the $+i$-eigenbundle of the generalized complex structure on the fiber $M_b$. Then $\tau(Z) : H^p_{\overline{\partial}}(M_b) \to H^{p+2}_{\overline{\partial}}(M_b)$ is the cup product with $\kappa_b(Z)$.
\end{proposition}
\begin{proof}
Since the statements to prove are local with respect to the base we may by restriction assume $B$ is contractible, choose a trivialization $M = B \times M_0$ and choose a closed $3$-form $H \in \Omega^3(M_0)$ representing the \v{S}evera class along the fibers. 

Let $a$ be a section of $H^p_{\overline{\partial}}$. Then we can find a differential form $\tilde{a}$ on $B \times M$ such that $\tilde{a}$ contracted with vector fields on $B$ is zero and such that the restriction of $\tilde{a}$ to the fiber $M_b$ is a $d_H$-closed section of $U_p$, hence $\partial$- and $\overline{\partial}$-closed as well.

Choose local coordinates $b^1, \dots , b^s$ on $B$ and consider the behavior of the generalized complex structure $J = J(b)$ around $b=0$. As a vector bundle the Courant algebroid $E$ is $TM_0 \oplus T^*M_0$ (pulled back to $B \times M_0$) and we can think of $J$ as a generalized complex structure on $TM_0 \oplus T^*M_0$ that varies with $b \in B$. We can expand $J(b)$ in the form $J(b) = J(0) + b^j J_j + \dots$, where the additional terms are higher order in $b$. Similarly the $+i$-eigenspace $L_b$ of $J(b)$ varies with $b$ and, as usual restricting the base if necessary, we can write $L_b$ as a graph $L_b = \{ X  + \epsilon(b)X \, | \, X \in L_0 \}$ where for each $b$, $\epsilon$ is a section of $\wedge^2 L_0^*$ on $M_0$. If we similarly expand $\epsilon$ like $\epsilon(b) = b^j \epsilon_j + \dots$ then $\epsilon_j$ is a representative for the generalized Kodaira-Spencer class $\kappa_0( \partial/ \partial b^j)$. By straightforward algebra one sees that
\begin{equation}\label{jjandks}
J_j = 2i \epsilon_j - 2i \overline{\epsilon_j}.
\end{equation}
Recall that $J$ can be viewed as a section of $\wedge^2 E$ and therefore has a Clifford action on differential forms. Under this action $U_p$ is the $-ip$-eigenspace of $J$ Now the fact that $\tilde{a}$ is a section of $U_p$ for all $b \in B$ gives us $J(b)\tilde{a}_b = -ip \tilde{a}_b$, where we write $\tilde{a}_b$ for the restriction of $\tilde{a}$ to the fiber over $b$. Let us now expand $\tilde{a}_b$ as $\tilde{a}_b = \tilde{a}_0 + b^j \tilde{a}_j + \dots$. Let us note that according to Equation (\ref{tgmformula}) we have $(\nabla_{\partial/\partial b^j} a)(0) = [\tilde{a}_j]$, where $\nabla$ is the twisted Gauss-Manin connection. Collecting $b^j$-coefficients we get
\begin{equation}\label{inftrel}
J_j \tilde{a}_0 + J(0) \tilde{a}_j = -ip \tilde{a}_j.
\end{equation}
We note that $\epsilon_j$ sends $U_p$ to $U_{p+2}$ and $\overline{\epsilon_j}$ sends $U_p$ to $U_{p-2}$. By (\ref{jjandks}) we have that $J_j \tilde{a}_0$ consists of a term in $U_{p+2}$ and a term in $U_{p-2}$. Feeding this in to (\ref{inftrel}) together with the fact that $J(0)$ acts like $-ik$ on $U_k$, we see that $\tilde{a}_j$ has only components of degrees $p-2,p$ and $p+2$. From this we immediately see that $\nabla$ sends $F^pH$ to $T^*B \otimes F^{p+2}H$. Next, to work out the induced map $\tau : F^{p}H / F^{p-2}H \to T^*B \otimes F^{p+2}H / F^{p}H$ we need only determine the degree $p+2$ part of $\tilde{a}_j$. From (\ref{jjandks}) and (\ref{inftrel}) we find that it is given by $\epsilon_j \tilde{a}_0$. Now since $[\tilde{a}_0] = a(0)$ and $\epsilon_j$ represents $\kappa_0(\partial/\partial b^j)$ we get that this equals $\kappa_0(\partial/\partial b^j) a(0)$ as claimed.
\end{proof}
\begin{remark}
From the above proof we also see that for any vector field $X$, $\nabla_X$ sends $H^p_{\overline{\partial}}$ to $H^{p-2}_{\overline{\partial}} \oplus H^p_{\overline{\partial}} \oplus H^{p+2}_{\overline{\partial}}$. The parts of $\nabla$ that raise or lower degree can be expressed using the generalized Kodaira-Spencer class and its conjugate.
\end{remark}

Let $(M,B,\pi,E,J)$ be a good family. We now consider period mappings for such families. Fix a basepoint $0 \in B$ and let $M_0 = \pi^{-1}(0)$. Fix an integer $p$ and consider the subbundle $F^pH \subseteq H^{p+n+\tau}_\mathbb{C}$. Let $f^p = {\rm dim}_{\mathbb{C}}(F^p H)$. Using the twisted Gauss-Manin connection $\nabla$ we can parallel translate the fibers of $F^pH$ back to the fiber of $H^{p+n+\tau}_\mathbb{C}$ at $0$. If $\tilde{B}$ is the universal cover of $B$ then this process of parallel translation determines a period mapping (c.f. \cite{gr2}):
\begin{equation*}
\mathcal{P}^{p} : \tilde{B} \to {\rm Grass}( f^p , H^{p+n+\tau}(M_0,H_0)_\mathbb{C} ),
\end{equation*}
where ${\rm Grass}( r , V )$ denotes the Grassmannian of $r$-dimensional complex subspaces of a complex vector space $V$.
\begin{remark}
That the period map $\mathcal{P}^{p}$ is defined on the universal cover $\tilde{B}$ in general rather than on $B$ itself reflects possible non-trivial monodromy of the twisted Gauss-Manin connection. Let $\rho : \pi_1(B) \to GL( H^*(M_0,H_0) )$ be the monodromy of the twisted Gauss-Manin connection. Then on viewing $\tilde{B}$ as a principal $\pi_1(B)$-bundle over $B$ we have that the period map satisfies $\mathcal{P}^{p}( x \gamma) = \rho(\gamma)^{-1} \mathcal{P}^{p}(x)$ for all $x \in \tilde{B}$ and $\gamma \in \pi_1(B)$.

We have considered here period maps obtained by considering  single term $F^p H$ of the filtration on twisted cohomology. By considering multiple terms we get period maps into more general flag varieties. However these maps are determined by the $\mathcal{P}^{p}$, so we restrict attention to these.
\end{remark}

Let $V$ be a complex vector space and $W$ be a complex $r$-dimensional subspace, so $W \in {\rm Grass}(r,V)$. Choose a complimentary subspace $W'$ so that $V = W \oplus W'$. Then an open neighborhood of $W$ in ${\rm Grass}(r,V)$ is given by those subspaces which can be written as a graph $\{ X + AX \, | \, X \in W \}$, where $A \in {\rm Hom}(W,W') \simeq {\rm Hom}(W,V/W)$. It follows that we can identify the tangent space $T_W {\rm Grass}(r,V)$ with ${\rm Hom}(W,V/W)$ and this identification is in fact independent of the choice of complement $W'$. Based on this identification we can describe the differential $d\mathcal{P}^{p}$ of the period map. Indeed the differential can be thought of as a map $\phi : TB \to {\rm Hom}( F^pH , H^{p+n+\tau}_\mathbb{C} / F^p H)$. To get the actual differential $d\mathcal{P}^{p}$ we use parallel translation to transport the various subspaces back to the basepoint. It is clear that $\phi$ can be described directly using the twisted Gauss-Manin connection as follows: if $s$ is a section of $F^p H$ and $X$ a vector field on $B$ then $\phi(X)s = \nabla_X s \, ({\rm mod} \, F^pH)$. Now as in Proposition \ref{grtrans} we know for each vector field $X$, $\phi(X)$ is actually valued in ${\rm Hom}( F^pH , F^{p+2}H / F^p H)$ and is induced from an element in ${\rm Hom}( F^pH / F^{p-2}H, F^{p+2} H / F^p H) = {\rm Hom}( H^p_{\overline{\partial}} , H^{p+2}_{\overline{\partial}} )$, namely the product with the generalized Kodaira-Spencer class. We have thus found an expression for the differential of the period maps.

\begin{proposition}\label{pmholo}
If the family of generalized complex structures $(E,J)$ is a holomorphic family then the period maps $\mathcal{P}^{p}$ are holomorphic.
\end{proposition}
\begin{proof}
We have seen how the differential of $\mathcal{P}^{p}$ is essentially given by the action $TB \otimes H^p_{\overline{\partial}} \to H^{p+2}_{\overline{\partial}}$ of the generalized Kodaira-Spencer class. In Proposition \ref{ksmc} we showed that for holomorphic families of generalized complex structures the generalized Kodaira-Spencer class is complex linear. It follows immediately that $\mathcal{P}^{p}$ is holomorphic.
\end{proof}


\section{Special cases}\label{speccas}


\subsection{Symplectic type}\label{symplectic}

Let $E$ be an exact Courant algebroid on $M$ and $J$ a generalized complex structure on $E$. Associated to $J$ is a bivector field $\beta \in \Gamma(M,\wedge^2 TM)$ given by letting $\beta \xi = \rho( J \xi)$ for any $1$-form $\xi$. As a consequence of the integrability of $J$ it turns out that $\beta$ is a Poisson structure \cite[Proposition 3.21]{gualgc}. We consider the two extreme cases. When $\beta$ has maximal rank $\beta : T^*M \to TM$ is an isomorphism and the inverse $\omega : TM \to T^*M$ is a symplectic structure, we say that $J$ is of symplectic type. When $\beta$ has minimal rank, that is $\beta = 0$ then $T^*M$ is invariant under $J$ and induces an integrable complex structure on $M$. We say that $J$ is of complex type in this case. In this section we consider the Hodge decomposition for symplectic type generalized complex structures. In Section \ref{cpxtype} we will consider complex type.\\

Let $J$ be a symplectic type generalized complex structure. Thus $\beta$ is a non-degenerate Poisson structure and can be inverted to a symplectic structure $\omega$. Choose an identification $E = TM \oplus T^*M$ with $H$-twisted Dorfman bracket for some closed $3$-form $H$. One can then show that $J$ has the form
\begin{equation*}
J = e^{B} \left[ \begin{matrix} 0 & -\omega^{-1} \\ \omega & 0 \end{matrix} \right] e^{-B} = \left[ \begin{matrix} \omega^{-1}B & -\omega^{-1} \\ \omega + B \omega^{-1} B & -B\omega^{-1} \end{matrix} \right],
\end{equation*}
where we write $\omega^{-1}$ for $\beta$ and $B$ is a $2$-form on $M$. To see this note that an isotropic lift of $TM$ is given by $J(T^*M)$. With respect to this lift $J$ has the form above with $B=0$. Any other splitting is related by a $2$-form $B$ and $J$ takes the form above. The corresponding pure spinor for $J$ is $\rho = e^{-B+i\omega}$. Integrability of $J$ actually forces $\rho$ to be $d_H$-closed which is equivalent to $d\omega = 0$ and $dB = H$. Note in particular that $H$ represents the trivial cohomology class. It is sufficient to assume that $H=0$, so that $B$ is a closed form.\\

To describe the decomposition of forms into the eigenspaces $U_k$ of $J$, let $\beta = \omega^{-1}$ act on forms by contraction $\beta : \wedge^k T^*M_{\mathbb{C}} \to \wedge^{k-2} T^*M_{\mathbb{C}}$. Then we have \cite{cav}
\begin{equation*}
U_{k-n} = \{ e^{-B+i\omega} e^{\beta/2i} \alpha \, | \, \alpha \in \wedge^k T^*M_{\mathbb{C}} \}.
\end{equation*}
If we define $\phi : \wedge^* T^*M \to \wedge^* TM$ by $\phi(\alpha) = e^{-B+i\omega} e^{\beta/2i} \alpha$ then $\phi$ restricted to $\wedge^k T^*M_{\mathbb{C}}$ is an isomorphism $\phi : \wedge^k T^*M \to U_{k-n}$. As in \cite{bryp} define an operator $\delta : \Gamma(\wedge^k T^*M_{\mathbb{C}}) \to \Gamma(\wedge^{k-1} T^*M_{\mathbb{C}})$ by setting $\delta = [\beta , d]$. We make use of the identity $[\beta , \delta] = 0$ from \cite{cav} to see that $[d , \beta^k ] = -k \beta^{k-1} \delta$. From this it is shown in \cite{cav} that
\begin{equation*}
d \phi(\alpha) = \phi( d\alpha - \frac{1}{2i} \delta \alpha)
\end{equation*}
and it immediately follows that
\begin{align*}
\overline{\partial} \phi(\alpha) &= \phi( d\alpha), \\
\partial \phi(\alpha) &= -\frac{1}{2i} \phi( \delta \alpha).
\end{align*}
We see in particular that $\phi$ induces an isomorphism $H^k(M,\mathbb{C}) \simeq H^{k-n}_{\overline{\partial}}(M)$ between de Rham cohomology with complex coefficients and the $\overline{\partial}$-cohomology. As a consequence the Fr\"olicher spectral sequence always converges at the $E_1$-stage in the symplectic case. However, this is not sufficient for the $\partial \overline{\partial}$-lemma to hold. Using the bundle isomorphism $\phi$ the operators $\overline{\partial},\partial$ correspond to $d$ and $-\delta/2i$. It follows that the $\partial \overline{\partial}$-lemma is equivalent to a corresponding property of the operators $d,\delta$. As explained in \cite{cav}, for a compact manifold $M$ it follows from \cite{mer} together with \cite{yan,mat} that the $\partial \overline{\partial}$-lemma is equivalent to the strong Lefschetz property: the map $\omega^{n-k} : H^k(M,\mathbb{R}) \to H^{2n-k}(M,\mathbb{R})$ is an isomorphism for $0 \le k \le n$, where $M$ has dimension $2n$.\\

Let $L \subset E_{\mathbb{C}}$ be as usual the $i$ eigenspace of $J$. If we set $\sigma = \omega + iB$ then we can write $L$ as a graph $L = \{ X - i\sigma X \, | \, X \in TM_{\mathbb{C}} \}$. There is an isomorphism $\psi : TM_{\mathbb{C}} \to L$ given by $\psi(X) = X - i\sigma X$. Observe that $TM_{\mathbb{C}}$ is a complex Lie algebroid with Lie bracket the usual Lie bracket of vector fields. One can easily see that $\psi$ is an isomorphism of Lie algebroids, indeed $\psi$ is just the inverse of the anchor $\rho : L \to TM_{\mathbb{C}}$. It follow that there is an induced pullback map $\psi^* : \wedge^* L^* \to \wedge^* T^*M_{\mathbb{C}}$ which induces an isomorphism of Lie algebroid cohomology groups $\psi^* : H^k(L) \to H^k(M,\mathbb{C})$. Let $\xi$ be a $1$-form. By the definition of $\psi$ it follows that $\psi^{-1}(\xi) = Y + i \sigma Y$, where $i\omega Y = \xi$. With some further algebra one deduces that the Clifford action $L^* \otimes U_k \to U_{k+1}$ corresponds under $\phi$ and $\psi$ to the wedge product $T^*M_{\mathbb{C}} \otimes \wedge^k T^*M_{\mathbb{C}} \to \wedge^{k+1} T^*M_{\mathbb{C}}$, namely
\begin{equation*}
\psi^{-1}(\xi) \phi(\alpha) = 2 \phi( \xi \wedge \alpha).
\end{equation*}
It follows that up to factors of $2$ the action $H^j(L) \otimes H^k_{\overline{\partial}}(M) \to H^{j+k}_{\overline{\partial}}(M)$ of the Lie algebroid cohomology of $L$ on the cohomology of $\overline{\partial}$ is just the wedge product in de Rham cohomology.\\

Let us assume now that $M$ is compact and the strong Lefschetz property holds. Thus we get a Hodge structure on the $\mathbb{Z}_2$-graded cohomology of $M$. The stability of the $\partial \overline{\partial}$-lemma under small deformations is particularly easy to see here, small deformations of symplectic structure will continue to satisfy the strong Lefschetz property. It is worth remarking here that all sufficiently small deformations of $M$ as a generalized complex manifold of symplectic type are also of symplectic type. The space of first order deformations is given by $H^2(L) \simeq H^2(M,\mathbb{C})$ and corresponds to deformations of $\sigma = \omega + iB$. In fact if $\mu$ is a closed complex $2$-form we consider the family of generalized complex structures corresponding to the family of spinors $\rho_t = e^{i(\sigma + t \mu)}$, which is a pure spinor for all small enough $t$. Differentiating at $t=0$ we get $i \mu \wedge e^{i\sigma}$ which immediately implies that under the identification $H^2(L) \simeq H^2(M,\mathbb{C})$, the generalized Kodaira-Spencer class is $i\mu/2$.
\begin{proposition}\label{symhs}
The subspace $F^p H(M)$ of the twisted cohomology group \linebreak $H^{p+n}(M,0)_{\mathbb{C}}$ is given by $e^{i\sigma} ( H^{p+n}(M,\mathbb{C}) \oplus H^{p+n-2}(M,\mathbb{C}) \oplus \dots )$.
\end{proposition}
\begin{proof}
Let $x$ be a class in $F^p H(M)$. This means that $x$ has a representative of the form $e^{i\sigma} e^{\beta/2i} \alpha$, where $\alpha$ is a sum of forms of degree $\le p+n$. Then $e^{-i\sigma}x$ is represented by $e^{\beta/2i}\alpha$. Note that since $\beta$ only lowers degree we have that $e^{\beta/2i}\alpha$ is a sum of closed forms of degree $\le p+n$. Thus $e^{-i\sigma} x \in H^{p+n}(M,\mathbb{C}) \oplus H^{p+n-2}(M,\mathbb{C}) \oplus \dots$. Conversely if $\mu$ is a sum of closed forms of degree $\le p+n$ then we can write $\mu = e^{\beta/2i}\alpha$ for some form $\alpha$ which also has degree $\le p+n$. It then follows that $e^{i\sigma}\mu = e^{i\sigma}e^{\beta/2i}\alpha$ is a class in $F^p H(M)$.
\end{proof}
We see from this result that in the symplectic case the filtration $F^pH(M)$ is completely determined from the cohomology ring $H^*(M,\mathbb{C})$ and the class $\sigma$. To determine how the groups $H^p_{\overline{\partial}}(M)$ sit inside $H^*(M,\mathbb{C})$ we can use the relation $H^p_{\overline{\partial}}(M) = F^p H(M) \cap \overline{ F^{-p} H(M)}$, so the entire Hodge decomposition follows from the structure of the cohomology ring and the class $\sigma$.\\

Suppose now we have a holomorphic family $\pi : M \to B$ of generalized complex structures of symplectic type satisfying the $\partial \overline{\partial}$-lemma. For each $b \in B$ we have a pure spinor $e^{i\sigma(b)}$ for $M_b$, where $\sigma = \omega + iB$ determines a section $[\sigma] = [\omega + iB]$ of the flat vector bundle $H^2_\mathbb{C}$ associated to the local system $R^2 \pi_* \mathbb{C}$. Since the family is holomorphic we have by Proposition \ref{pmholo} that the period maps are holomorphic, but from Proposition \ref{symhs} we see that the period maps are holomorphic if and only $\sigma$ is holomorphic as a section of $H^2_\mathbb{C}$.


\subsection{Complex type}\label{cpxtype}

On a generalized complex manifold $(M,E,J)$ the decomposition of forms determined by $J$ gives rise to a grading on twisted cohomology and a corresponding filtration $F^j H(M)$. However there is another filtration $W^j H(M)$ on twisted cohomology which does not involve $J$. Namely it is the filtration induced by a corresponding filtration $\{W^j \}$ of the bundle of even or odd differential forms, where $W^j = \bigoplus_{p \ge j} \wedge^p T^*M_\mathbb{C}$. This is a filtration on the total space $S_\mathbb{C} = \wedge^* T^*M_\mathbb{C}$ of differential forms. Since the twisted differential $d_H = d + H$ only increases degree it preserves this filtration. We let $W^j H(M)$ be the image of the $d_H$-cohomology of $W^j$ in the total twisted cohomology group $H^*(M,H)_\mathbb{C} = H^0(M,H)_\mathbb{C} \oplus H^1(M,H)_\mathbb{C}$ induced by the inclusion $W^j \to S_\mathbb{C}$. 

On a generalized complex manifold we thus have two filtrations on twisted cohomology. However it would appear that in general these two filtrations do not interact well, or at least it is not clear why they should. The problem is that while the twisted differential $d_H$ respects the filtration $\{ W^j \}$, the same need not be true for the operators $\partial,\overline{\partial}$. Since $\partial,\overline{\partial}$ are obtained from $d_H$ through the inclusion $T^*M \to E$ followed by projection to the $\pm i$ eigenspaces $L,\overline{L}$ it follows that $\partial,\overline{\partial}$ will preserve the filtration $\{ W^j \}$ if and only if we have
\begin{equation*}
T^*M_{\mathbb{C}} = (T^*M_{\mathbb{C}} \cap L ) \oplus (T^*M_{\mathbb{C}} \cap \overline{L} ).
\end{equation*}
Equivalently, this says that $T^*M$ as a subbundle of $E$ is preserved by $J$. In this case $J$ induces an integrable complex structure on $M$ and we say $J$ is of complex type. Let $I$ denote the induced complex structure on $M$. One can show that there exists an isotropic splitting $E = TM \oplus T^*M$ for $E$ such that $J$ has the form
\begin{equation}\label{gcsct}
J = \left[ \begin{matrix} -I & 0 \\ 0 & I^* \end{matrix} \right].
\end{equation}
Note however that when considering families it is better to think of an arbitrary generalized complex structure of complex type to look like a $B$-shift of (\ref{gcsct}), since we can always deform such a $J$ by a family of $B$-shifts. Let $H$ denote the closed $3$-form corresponding to the choice of splitting for $E$. Then integrability of $J$ is equivalent to integrability of $I$ together with the restriction that $H$ be of type $(2,1) + (1,2)$. The $i$ eigenspace $L$ is given by $L = T^{0,1}M \oplus {T^*}^{1,0}M$. We then have decompositions
\begin{align*}
\wedge^k L^* &= \bigoplus_{p+q=k} \wedge^p T^{1,0}M \otimes \wedge^{0,q} T^*M, \\
U_k &= \bigoplus_{q-p = k} \wedge^{p,q} T^*M.
\end{align*}
Let us write $H = h + \overline{h}$, where $h$ is a $(1,2)$-form. Then the differential $\overline{\partial}$ is easily seen to have the form $\overline{\partial} = \overline{\partial}_0 + \overline{h}$, where $\overline{\partial}_0$ denotes the usual (untwisted) $\overline{\partial}$-operator on a complex manifold. Thus $H^k_{\overline{\partial}}(M)$ is a kind of twisted Dolbeault cohomology. For any holomorphic vector bundle $A$ we can combine the twisted $\overline{\partial}$ operator $\overline{\partial} : \Gamma(U_k) \to \Gamma(U_{k+1})$ with the $\overline{\partial}$ operator for $A$ to get an operator $\overline{\partial}_A : \Gamma(A_k) \to \Gamma(A_{k+1})$ where $A_k = A \otimes U_k$. We let $H^k_{\overline{\partial}}(A)$ denote the cohomology groups. In particular we have that the Lie algebroid cohomology has this form: $H^k(L) \simeq H^k_{\overline{\partial}}(K^*)$, where $K = \wedge^{n,0}T^*M$ is the canonical bundle for the complex structure on $M$.\\

In what follows we would like to compare the filtration $F^k H(M)$ associated to the generalized Fr\"{o}licher spectral sequence with the with the filtration $W^j H(M)$ defined as above. It will be convenient to temporarily forget the $\mathbb{Z}_2$-grading on twisted cohomology, thus instead of the usual filtrations $\{ F^kH(M) \}$ on even and odd twisted cohomology, let us instead consider the wrapped up filtration $\{ \tilde{F}^k H(M) \}$ on the total space of twisted cohomology defined by $\tilde{F}^k H(M) = F^kH(M) \oplus F^{k-1}H(M)$. Clearly we have
\begin{equation*}
\tilde{F}^{-n}H(M) \subseteq \tilde{F}^{-n+1}H(M) \subseteq \cdots \subseteq \tilde{F}^nH(M) = H^*(M,H)_\mathbb{C}.
\end{equation*}
Set $\tilde{F}^k = \bigoplus_{j \ge k} U_k$. It follows that $\tilde{F}^kH(M)$ is the subspace of twisted cohomology which have representatives in $\tilde{F}^k$. As above let $W^j = \bigoplus_{ i \ge j} \wedge^i T^*M_{\mathbb{C}}$ and observe the following properties:
\begin{itemize}
\item{$\overline{W^j} = W^j$}
\item{$W^j = \bigoplus_k (U_k \cap W^j)$}
\item{$\overline{\partial}( \Gamma(W^j) ) \subseteq \Gamma( W^j)$.}
\end{itemize}
In fact $\overline{\partial}$ sends $W^j$ to $W^{j+1}$, but in what follows we need only the weaker statement above.
\begin{proposition}\label{mhs}
Let $W^jH(M)$ be the filtration on twisted cohomology induced by the filtration $\{W^j\}$ and $\tilde{F}^jH(M)$ the filtration on twisted cohomology as above. If $(M,E,J)$ satisfies the $\partial \overline{\partial}$-lemma then the filtrations $\{W^j H(M)\},\{\tilde{F}^j H(M)\}$ form a mixed Hodge structure on the twisted cohomology groups in the following sense. Let $\tilde{F}^{i,j}H(M)$ be the image of $\tilde{F}^i H(M) \cap W^jH(M)$ under the projection $W^jH(M) \to W^j H(M) / W^{j+1} H(M)$. Then the $\tilde{F}^{i,j}H(M)$ define a Hodge structure on $W^jH(M)/W^{j+1}H(M)$.
\end{proposition}
\begin{proof}
The result will follow by considering a certain class of complex. Let $V$ be a complex vector space with differential $d$ and grading in the range $-n, \dots , n$, so $V = V_{-n} \oplus \dots \oplus V_n$. Assume also that $\overline{V_k}  = V_{-k}$, that $d = \partial + \overline{\partial}$ with $\overline{\partial}$ having degree $1$ and suppose also that the $\partial \overline{\partial}$-lemma holds. Next let $W^j$ be a decreasing filtration on $V$ with the following properties:
\begin{itemize}
\item[(W1)]{$\overline{W^j} = W^j$}
\item[(W2)]{$W^j = \bigoplus_k (V_k \cap W^j)$}
\item[(W3)]{$\overline{\partial} W^j  \subseteq W^j$.}
\end{itemize}
Let $F^k = \bigoplus_{j \le k} V_k$ and $F^k H, W^j H$ the filtrations on the cohomology of $V$ induced by $\{F^k\},\{W^j\}$. We will show that $F^kH, W^j H$ defines a mixed Hodge structure on $H$, the cohomology of $V$. The first step is to prove the following:
\begin{equation*}
F^kH \cap W^j H = \{ [x] \, | \, x \in F^k \cap W^j, \, \overline{\partial}x = 0\},
\end{equation*}
where $[x]$ denotes the $\overline{\partial}$ cohomology class of $x$ and we are making use of the $\partial \overline{\partial}$-lemma to deduce that the $d$ and $\overline{\partial}$ cohomology of $V$ are isomorphic. To prove this suppose that $x \in F^k$ and $y \in W^j$ are $\overline{\partial}$-closed and $[x] = [y]$. Then $y = x + \overline{\partial}z$ for some $z \in V$. Write $z = \sum_i z_i$ where $z_i \in V_i$. Set $z' = \sum_{i \ge k-1} z_i$ so that $z' \in F^{k-1}$ and let $x' = x + \overline{\partial}z'$. Then $[x'] = [x]$ and $x' \in F^{k}$. Also write $y = y' + y''$ where $y' \in F^k$ and $y'' \in \overline{F^{-k-2}}$. Then by (W2) we have that $y',y'' \in W^j$ and it is straightforward to see that $x' = y'$. So $x' \in F^k \cap W^j$ and represents the same class as $x$ and $y$ as required.\\

Let $\tilde{F}^{j,k}H$ be the image of $F^jH \cap W^k H$ under the projection $W^k H \to W^k H / W^{k+1}H$. It follows that $\tilde{F}^{j,k} = (F^jH \cap W^k H )/( F^jH \cap W^{k+1}H)$ and thus 
\begin{equation*}
\tilde{F}^{j,k}H = \frac{\{ [x] \, | \, x \in F^j \cap W^k, \, \overline{\partial}x = 0 \}}{\{ [x] \, | \, x \in F^j \cap W^{k+1}, \, \overline{\partial}x = 0 \}}.
\end{equation*}
Now using this we will show that $\tilde{F}^{j,k}H \oplus \overline{\tilde{F}^{-j-2,k}} = H$, which is the condition for the $\tilde{F}^{j,k}$ to define a Hodge structure. First we show that $\tilde{F}^{j,k}H \oplus \overline{\tilde{F}^{-j-2,k}} \to H$ is injective. Indeed this follows since $(F^j H \cap W^kH) \cap (\overline{F^{-j-2}H} \cap W^kH) \subseteq F^j H \cap \overline{F^{-j-2}H} = 0$, since the $F^jH$ define a Hodge structure on $H$. To show surjectivity let $x \in W^k$ be $\overline{\partial}$-closed. Then we may write $x = x' + x''$ where $x' \in F^j$ and $x'' \in \overline{ F^{-j-2}}$. By (W2) we also have $x',x'' \in W^k$ and it is clear that $\overline{\partial}x' = \overline{\partial}x'' = 0$. We thus have $[x] = [x'] + [x'']$ where $[x'] \in \tilde{F^{j,k}}$ and $[x''] \in \overline{\tilde{F}^{-j-2,k}}$.
\end{proof}
We would like to consider variations of Hodge structure in generalized complex families of complex type. Consider the Lie algebroid cohomology $H^2(L)$. There is a natural projection map $H^2(L) \to H^0( M , \wedge^2 T^{1,0}M)$ obtained by projection to the $\wedge^2 T^{1,0}M$ component of $\wedge^2 L^*$. If we restrict to deformations of generalized complex structure that preserve the property of being of complex type, we see that this component must vanish. Let $H^2_0(L)$ denote the kernel of the map $H^2(L) \to H^0( M , \wedge^2 T^{1,0}M)$. Associated to the short exact sequence $\wedge^{1,0}T^*M \to L \to T^{0,1}M$ is a spectral sequence for the Lie algebroid cohomology $H^*(L)$. We find that $H^2_0(L)$ fits into an exact sequence
\begin{equation*}
H^0(M,T^{1,0}M) \buildrel \overline{h} \over \to H^2(M,\mathcal{O}) \to H^2_0(L) \buildrel \rho \over \to H^1( M , T^{1,0}M) \buildrel \overline{h} \over \to H^3(M,\mathcal{O}).
\end{equation*}
The map $\rho : H^2_0(L) \to H^1( M , T^{1,0}M)$ sends the deformation of generalized complex structure of complex type to the Kodaira-Spencer class representing the deformation of underlying complex structure. The map $H^2(M,\mathcal{O})$ represents a shift by a $2$-form $B + \overline{B}$, where $B$ is a $\overline{\partial}$-closed $(0,2)$-form.\\

Let us explain how the variation of the Hodge structures induced by the mixed Hodge structure only see the $H^1(M , T^{1,0}M)$ part of the generalized Kodaira-Spencer class. In other words, if we take the Hodge structures $\tilde{F}^{i,j}H(M)$ induced by the weight filtration, their infinitesimal variation only depends on the image of the generalized Kodaira-Spencer class under the map $\rho : H^2_0(L) \to H^1( M , T^{1,0}M)$. To see this, one need only observe that there is a natural filtration on the bundles $\wedge^k L^*$ given by setting $F^p \wedge^k L^* = \bigoplus_{i \ge p} \wedge^{0,i} T^*M \otimes \wedge^{k-i} T^{1,0}M$. Then one can show that the Clifford action of $F^p \wedge^k L^*$ sends $U_i \cap W^j$ to $U_{i+k} \cap W^{j+2p-k}$. In particular we see that $F^1 \wedge^2 L^*$ sends $U_i \cap W^j$ to $U_{i+2} \cap W^j$, while $F^p \wedge^2 L^*$ sends $U_i \cap W^j$ to $U_{i+2} \cap W^{j+2}$. As a consequence we get an induced action $T^{1,0}M \otimes \wedge^{0,1}T^*M = F^1 \wedge^2 L^* / F^2 \wedge^2 L^*$ sending $U_i \otimes W^j/W^{j+1}$ to $U_{i+2} \otimes W^{j}/W^{j+1}$. Therefore on passing to the $W^j/W^{j+1}$ quotients the $\wedge^{0,2}T^*M$ part of the generalized Kodaira-Spencer class gets projected out.


\subsection{Generalized Calabi-Yau manifolds}\label{gcym}

Let $E = TM \oplus T^*M$ with $H$-twisted Dorfman bracket. A generalized complex manifold $(M,J,H)$ is called a {\em generalized Calabi-Yau manifold} \cite{hit} if there exists a globally defined $d_H$-closed pure spinor $\rho$ for $J$. 
\begin{proposition}
There is an isomorphism $H^k(L) \simeq H^{k-n}_{\overline{\partial}}(M)$ given by sending a class $a \in H^k(L)$ to $a \rho$. Under this isomorphism the Clifford action $H^j(L) \otimes H^k_{\overline{\partial}}(M) \to H^{j+k}_{\overline{\partial}}(M)$ corresponds to the wedge  product $H^j(L) \otimes H^{k+n}(L) \to \linebreak H^{j+k+n}(L)$.
\end{proposition}
\begin{proof}
The map $f: \wedge^k L^* \to U_{k-n}$ defined by $f(a) = a \rho$ is easily seen to be an isomorphism and moreover $f(d_L a) = \overline{\partial} f(a)$, since $\rho$ is closed. The second part of the proposition easily follows since $a f(b) = f(a\wedge b)$, for $a,b$ sections of $\wedge^* L^*$.
\end{proof}

For compact generalized Calabi-Yau manifolds which satisfy the $\partial \overline{\partial}$-lemma there is a Tian-Todorov theorem \cite{got} which states that all infinitesimal deformations of the generalized complex structure are unobstructed and there is a smooth local moduli space of generalized Calabi-Yau strutcures on $M$. Note however that there could potentially be deformations of $M$ for which the $\partial \overline{\partial}$-lemma does not hold, so there is not necessarily a smooth global moduli space of generalized Calabi-Yau structures on $M$. Instead we let $\mathcal{M}$ denote the space of generalized Calabi-Yau structures on $M$ for which the $\partial \overline{\partial}$-lemma holds. Then we know that $\mathcal{M}$ has the structure of a smooth manifold.\

At the point of $\mathcal{M}$ corresponding to a pure spinor $\rho$ of parity $\tau$ we have that the tangent space $T_\rho \mathcal{M}$ canonically identifies with the Lie algebroid cohomology $H^2(L)$, where $L$ is the $i$ eigenspace of the generalized complex structure associated to $\rho$. Assume $\mathcal{M}$ is connected, or restrict to the connected component containing $\rho$. Let $\tilde{\mathcal{M}}$ be the universal covering of $\mathcal{M}$. On fixing the basepoint $\rho \in \mathcal{M}$ we get a period mapping $\mathcal{P}^{-n} : \tilde{\mathcal{M}} \to \mathbb{P}( H^\tau(M , H)_{\mathbb{C}})$ which for a generalized Calabi-Yau corresponding to a pure spinor $\rho'$ assigns the complex line in $ H^\tau(M , H)_{\mathbb{C}}$ spanned by $\rho'$.
\begin{proposition}\label{pmimm}
The period map $\mathcal{P}^{-n} : \tilde{\mathcal{M}} \to \mathbb{P}( H^\tau(M , H)_{\mathbb{C}})$ is an immersion.
\end{proposition}
\begin{proof}
By Proposition \ref{grtrans}, the differential of $\mathcal{P}^{-n}$ at a point corresponding to a generalized Calabi-Yau structure $\rho$ is given by the map $H^2(L) \to H^{2-n}_{\overline{\partial}}(M)$ which sends a class $a \in H^2(L)$ to $a \rho \in H^{2-n}_{\overline{\partial}}(M)$, followed by the inclusion of $H^{2-n}_{\overline{\partial}}(M)$ in $H^\tau(M,H)_{\mathbb{C}}/ H^{-n}_{\overline{\partial}}(M) = H^{2-n}_{\overline{\partial}}(M) \oplus H^{4-n}_{\overline{\partial}}(M) \oplus \dots \oplus H^{n}_{\overline{\partial}}(M)$ which is clearly injective.
\end{proof}


\section{Variation of Hodge structure for generalized K\"ahler manifolds}\label{vhsgkm}

\subsection{Generalized K\"ahler geometry}\label{gkg}

Let $E$ be a Courant algebroid. A {\em generalized metric} on $E$ is an endomorphism $G : E \to E$ such that $G$ preserves the pairing $\langle \, , \, \rangle$ in the sense that $\langle Ga , Gb \rangle = \langle a , b \rangle$ and the induced bilinear form $\langle Ga , b \rangle$ is symmetric and positive definite. Note that by these identities we have that $G^2=1$. We can then decompose $E$ into the $1$ and $-1$ eigenspaces $V^{\pm}$ of $G$. It follows that $V^+$ is a maximal positive definite subspace of $E$ with respect to $\langle \, , \rangle$ and $V^-$ is a maximal negative definite subspace, the orthogonal complement of $V^+$. Suppose that $J$ is a generalized complex structure on $E$. We say that $J$ and $G$ are {\em compatible} if they commute as endomorphisms of $E$; $JG = GJ$. In this case the metric associated to $G$ is hermitian with respect to $J$, that is $\langle GJa ,Jb \rangle = \langle Ga,b \rangle$.

Let $J_1,G$ be a generalized complex structure and compatible generalized metric. Since $J_1$ and $G$ commute we find that $J_2 = GJ_1$ is a generalized almost complex structure. In general $J_2$ will not be integrable. We say that the pair $J_1,G$ defines a {\em generalized K\"ahler} structure \cite{gual} if $J_2$ is integrable. Equivalently a generalized K\"ahler structure is a pair of commuting generalized complex structures $J_1,J_2$ such that $G = -J_1J_2$ is a generalized metric.\\

Let $E$ be an exact Courant algebroid over a manifold $M$ of dimension $2n$. Choose a closed $3$-form $H$ on $M$ so that we can identify $E$ with $TM \oplus T^*M$ with $H$-twisted bracket. Suppose that $J_1,J_2,G$ defines a generalized K\"ahler structure on $E$. We note that the bundle of differential forms $S$ on $M$ has two decompositions corresponding to the eigenspaces of the Clifford actions of $J_1,J_2$ on $S$. Since $J_1,J_2$ commute there exists a simultaneous eigenspace decomposition. We let $U_{r,s}$ denote the subspace of $S$ which is the $-ir$ eigenspace of $J_1$ and $-is$ eigenspace of $J_2$. To better understand this decomposition let us first note that there is a related decomposition of $E_{\mathbb{C}}$, namely
\begin{equation*}
E_{\mathbb{C}} = L_1^+ \oplus L_1^- \oplus \overline{L_1^+} \oplus \overline{L_1^-}.
\end{equation*}
For $i=1,2$ let $L_i$ denote the $i$ eigenspace of $J_i$. Then $L_1^+$ is by definition $L_1 \cap L_2$ and $L_1^-$ is $L_1 \cap \overline{L_2}$. It follows that $L_1 = L_1^+ \oplus L_1^-$ and $L_2 = L_1^+ \oplus \overline{L_1^-}$. Additionally it is clear that $V^+_{\mathbb{C}} = V^+ \otimes \mathbb{C} = L_1^+ \oplus \overline{L_1^+}$, where $V^+$ is the $1$ eigenspace of $G$, similarly $V^-_{\mathbb{C}} = L_1^- \oplus \overline{L_1^-}$. The fact that $J_1$ and $J_2$ are both integrable implies that the subbundles $L_1^+,L_1^-$ are integrable in the sense that their spaces of sections are closed under the Courant bracket.

The Clifford action of $L_1^+$ sends $U_{r,s}$ to $U_{r-1,s-1}$ and the Clifford action of $L_1^-$ sends $U_{r,s}$ to $U_{r-1,s+1}$. It follows that we may identify $U_{r,s}$ with $\wedge^p (L_1^+)^* \otimes \wedge^q (L_1^-)^* \otimes U_{-n,0}$ where $r = p+q - n$, $s = p-q$. Note in particular that $U_{r,s}$ is zero unless $r+s = n \, ({\rm mod} \, 2)$. It also follows that the twisted differential $d_H$ can be decomposed into four components:
\begin{equation*}
d_H = \delta_+ + \delta_- + \overline{\delta_+} + \overline{\delta_-},
\end{equation*}
where the bi-degrees of $\delta_+,\delta_-,\overline{\delta_+},\overline{\delta_-}$ are $(-1,-1),(-1,1),(1,1),(1,-1)$ respectively. For $i=1,2$, let $\partial_i,\overline{\partial_i}$ be the $\partial$ and $\overline{\partial}$ operators for $J_i$. Then we clearly have $\overline{\partial_1} = \overline{\delta_+} + \overline{\delta_-}$ and $\overline{\partial_2} = \overline{\delta_+} + \delta_-$.

The generalized metric $G$ induces a hermitian metric on the bundle $S_{\mathbb{C}}$ of complex forms and taking $M$ to be compact, we get a corresponding $L^2$ metric on the space of sections of $S_{\mathbb{C}}$. We may then define formal adjoints $\delta_+^*,\delta_-^*$ for $\delta_+,\delta_-$. From \cite{gualhod} we have the following relations, known as the generalized K\"ahler identites:
\begin{align*}
\delta_+^* = &-\overline{\delta_+} \\
\delta_-^* = & \; \; \overline{\delta_-}.
\end{align*}
To the operators $\overline{\delta_-},\overline{\delta_+},\overline{\partial_1},\overline{\partial_2},d_H$ we may define associated Laplacians $\linebreak \Delta_{\overline{\delta_-}},\Delta_{\overline{\delta_+}},\Delta_{\overline{\partial_1}},\Delta_{\overline{\partial_2}},\Delta_{d_H}$. As a consequence of the generalized K\"ahler identites we get that the Laplacians coincide up to constant factors:
\begin{equation*}
\Delta_{\overline{\delta_-}} = \Delta_{\overline{\delta_+}} = 2\Delta_{\overline{\partial_1}} = 2\Delta_{\overline{\partial_2}} = 4\Delta_{d_H}.
\end{equation*}
This has a number of implications for the twisted cohomology of a compact generalized K\"ahler manifold. First of all it implies that the $\partial \overline{\partial}$-lemma holds for both $J_1$ and $J_2$. Second it implies that there is a bigraded decomposition of twisted cohomology. Let $H^{r,s}_{\overline{\delta_+}}(M)$ denote the degree $(r,s)$ $\overline{\delta_+}$ cohomology. It follows that
\begin{align*}
H^*(M,H)_\mathbb{C} =& \bigoplus_{r+s = n \, ({\rm mod} \, 2)} H^{r,s}_{\overline{\delta_+}}(M) \\
H^k_{\overline{\partial_1}}(M) =& \bigoplus_{k+s = n \, ({\rm mod} \, 2)} H^{k,s}_{\overline{\delta_+}}(M) \\
H^k_{\overline{\partial_2}}(M) =& \bigoplus_{r+k = n \, ({\rm mod} \, 2)} H^{r,k}_{\overline{\delta_+}}(M).
\end{align*}


\subsection{Lie algebroid decompositions}\label{lad}

If $(M,E,J_1,J_2)$ is a generalized K\"ahler manifold then we have seen that the Lie algebroids $L_1,L_2$ corresponding to $J_1,J_2$ admit decompositions $L_1 = L_1^+ \oplus L_1^-$, $L_2 = L_1^+ \oplus \overline{L_1^-}$. Additionally $L_1^+,L_1^-$ and their conjugates are themselves Lie algebroids with bracket and anchor inherited from their inclusion in the complexification $E_{\mathbb{C}}$ of the Courant algebroid $E$. Let $\rho_1^+ : L_1^+ \to TM_{\mathbb{C}}$ denote the anchor for $L_1^+$, which by definition is just the restriction of $\rho$, the anchor of $E_{\mathbb{C}}$ to $L_1^+$. Similarly let $\rho_1^-,\overline{\rho_1^+},\overline{\rho_1^-}$ denote the anchors on $L_1^-,\overline{L_1^+},\overline{L_1^-}$. The direct sum decomposition $L_1 = L_1^+ \oplus L_1^-$ has the properties that it is a decomposition into subalgebras and the anchors are related by $\rho_1(a \oplus b) = \rho_1^+(a) + \rho_1^-(b)$, where $\rho_1$ is the anchor for $L_1$. The corresponding properties also apply to the decomposition $L_2 = L_1^+ \oplus \overline{L_1^-}$. We would like to consider such decompositions of Lie algebroids more abstractly and examine the consequences for the Lie algebroid cohomology.\\

Let $(A,\rho)$ be a Lie algebroid and suppose that there are Lie algebroids $(A_1,\rho_1)$, $(A_2,\rho_2)$ with the properties that as vector bundles $A = A_1 \oplus A_2$, the spaces $\Gamma(A_1),\Gamma(A_2)$ are subalgebras of $\Gamma(A)$ and the anchors are related by $\rho = \rho_1 + \rho_2$. We say that $A = A_1 \oplus A_2$ is a Lie algebroid decomposition. The graded differential complex $(\wedge^* A^*, d_A)$ admits a bigrading by setting $\wedge^{p,q} A^* = \wedge^p A_1^* \otimes \wedge^q A_2^*$. Let $a_1$ be a section of $A_1$ and $a_2$ a section of $A_2$. We then define $a_1 \cdot a_2$ to be the projection of $[a_1,a_2]$ to $A_2$. Similarly define $a_2 \cdot a_1$ to be the projection of $[a_2,a_1]$ to $A_1$. The Lie algebroid identities for $A$ immediately imply the following identities:
\begin{align*}
(fa_1) \cdot (a_2) &= f (a_1 \cdot a_2) \\
(a_1)\cdot (fa_2) &= f(a_1 \cdot a_2) + \rho_1(a_1)(f)a_2 \\
\left[ a_1,b_1 \right] \cdot a_2 &= a_1 \cdot (b_1 \cdot a_2) - b_1 \cdot (a_1 \cdot a_2),
\end{align*}
where $a_1,b_1$ are sections of $A_1$, $a_2$ is a section of $A_2$ and $f$ a function. These identities say that $A_2$ has the structure of a Lie algebroid module for $A_1$. Similarly we get that $A_1$ is a Lie algebroid module for $A_2$. Since $A_2$ is a Lie algebroid module for $A_1$ we get an induced Lie algebroid module structure on $\wedge^q A_2^*$. In fact the module structure is as follows:
\begin{equation*}
(X \cdot \alpha)(Y_1, \dots , Y_q) = \rho_1(X)\alpha(Y_1,\dots,Y_q) + \sum_{j=1}^q (-1)^j \alpha( X \cdot Y_j,X_1,\dots,\hat{X_j},\dots,X_q),
\end{equation*}
where $X \in \Gamma(A_1)$, $Y_1,\dots,Y_q \in \Gamma(A_2)$ and $\alpha \in \Gamma(\wedge^q A_2^*)$. If $W$ is any Lie algebroid module for $A_1$ we get a corresponding differential graded complex $\Gamma(W) \buildrel d_{A_1} \over \to \Gamma(W \otimes A_1^*) \buildrel d_{A_1} \over \to \Gamma(W \otimes \wedge^2 A_1^*) \to \dots$, in particular this applies to $W = \wedge^q A_2^*$ so we get an operator $d_{A_1} : \Gamma( \wedge^p A_1^* \otimes \wedge^q A_2^*) \to \Gamma( \wedge^{p+1} A_1^* \otimes \wedge^q A_2^*)$. Similarly by swapping $A_1$ and $A_2$ we get an operator $d_{A_2} : \Gamma( \wedge^p A_1^* \otimes \wedge^q A_2^*) \to \Gamma( \wedge^{p} A_1^* \otimes \wedge^{q+1} A_2^*)$.
\begin{proposition}
For any Lie algebroid decomposition $A = A_1 \oplus A_2$ we have $d_A = d_{A_1} + d_{A_2}$. Moreover $d_{A_1}d_{A_2} + d_{A_2}d_{A_1} = 0$.
\end{proposition}
\begin{proof}
Let $\xi$ be a section of $A_1^*$. It is a simple calculation to see that $d_A(\xi) = d_{A_1}\xi + d_{A_2}\xi$. Similarly one can check this identity when $\xi$ is a section of $A_2^*$. Suppose again that $\xi$ is a section of $A_1$. Then for any $\alpha \in \Gamma( \wedge^{p,q} A^*)$ we have
\begin{align*}
d_{A_1}( \xi \wedge \alpha) &= d_{A_1}(\xi) \wedge \alpha - \xi \wedge d_{A_1} \alpha \\
d_{A_2}( \xi \wedge \alpha) &= d_{A_2}(\xi) \wedge \alpha - \xi \wedge d_{A_2} \alpha.
\end{align*}
The second of these identities is essentially the fact that the wedge product $A_1^* \otimes \wedge^k A_1^* \to \wedge^{k+2} A_1^*$ is a morphism of Lie algebroid modules for $A_2$. Adding these we get
\begin{equation}\label{leib}
(d_{A_1} + d_{A_2})( \xi \wedge \alpha) = d_A(\xi) \wedge \alpha - \xi \wedge (d_{A_1} + d_{A_2}) \alpha.
\end{equation}
The same argument can be applied when $\xi$ is a section of $A_2^*$. Next we observe that the identity $d_A = d_{A_1} + d_{A_2}$ is clearly true for sections of $\wedge^{0,0} A^*$ since $\rho = \rho_1 + \rho_2$. From here we proceed by induction using (\ref{leib}) to see that $d_A = d_{A_1} + d_{A_2}$ in all degrees. The identity $d_{A_1}d_{A_2} + d_{A_2}d_{A_1}$ easily follows since $d_A^2 = d_{A_1}^2 = d_{A_2}^2 = 0$.
\end{proof}


\subsection{Generalized K\"ahler deformations}\label{gkd}

We would now like to consider families of generalized K\"ahler structures and the related deformation theory. For generalized K\"ahler structures there is a deformation complex describing infinitesimal deformations, but it is not elliptic except when the manifold is two dimensional. Therefore even on a compact manifold we do not expect a finite dimensional space of infinitesimal deformations. This phenomenon also occurs for ordinary K\"ahler manifolds. For example if $(M,I,\omega)$ is a K\"ahler manifold then for all functions $\phi$ sufficiently small in an appropriate sense, we get a new K\"ahler structure by replacing the K\"ahler form $\omega$ with $\omega + i\partial \overline{\partial} \phi$. From this one sees that K\"ahler structures admit many non-trivial deformations and we expect similar behavior from generalized K\"ahler structures. Nevertheless we now define smooth families of generalized K\"ahler manifolds.
\begin{definition}
Let $X \to B$ be a fiber bundle, $E \to B$ a family of exact Courant algebroids over $B$ obtained by reduction of an exact Courant algebroid $F$ on $X$. A generalized K\"ahler structure $J_1,J_2$ on $E$ will be called a {\em smooth family of generalized K\"ahler manifolds}. 
\end{definition}
One can also consider families of generalized K\"ahler manifolds where one of the complex structures varies in a holomorphic family.\\

As usual, by locally trivializing the family we can view such a family as a fixed manifold $M$ and closed $3$-form $H$, but with $J_1,J_2$ varying. Small deformations of $J_1,J_2$ can be expressed as in Section \ref{gksc} by taking the $i$ eigenspaces to be graphs over $L_1,L_2$ as in (\ref{ieigdef}). Thus there exists sections $\epsilon_1 \in \Gamma( \wedge^2 L_1^*)$, $\epsilon_2 \in \Gamma( \wedge^2 L_2^*)$ such that the corresponding $i$ eigenspaces are
\begin{align*}
L'_1 &= \{ X + \epsilon_1 X \, | \, X \in L_1 \} \\
L'_2 &= \{ X + \epsilon_2 X \, | \, X \in L_2 \}.
\end{align*}
Let $J_1',J_2'$ denote the corresponding generalized almost complex structures. If $J_1',J_2'$ commute then provided $\epsilon_1,\epsilon_2$ are sufficiently small we will have that $G' = -J_1'J_2'$ is a generalized metric. The requirement that $J_1',J_2'$ commute imposes a non-linear compatibility condition on $\epsilon_1,\epsilon_2$. Instead of working out this full condition we will consider only infinitesimal deformations and the compatibility condition will be linearized. Thus we suppose $\epsilon_i = \epsilon_i(t)$ depends smoothly on a parameter $t$ such that $\epsilon_i(0) = 0$ and we differentiate at $t=0$. We have
\begin{equation*}
J'_i(t) = J_i + t(2i \epsilon_i'(0) - 2i \overline{\epsilon'_i(0)}) + O(t^2).
\end{equation*}
Now since $\wedge^2 L_1^* = \wedge^2 (L_1^+)^* \oplus \left( (L_1^+)^* \otimes (L_1^-)^* \right) \oplus \wedge^2 (L_1^-)^*$ we may write $\epsilon_1'(0) = (\epsilon_1^{(1)},\epsilon_1^{(2)},\epsilon_1^{(3)})$ where $\epsilon_1^{(1)} \in \Gamma( \wedge^2 (L_1^+))$, $\epsilon_1^{(2)} \in \Gamma( (L_1^+)^* \otimes (L_1^-)^*)$, $\epsilon_1^{(3)} \in \Gamma( \wedge^2 (L_1^-)^*)$. Similarly replacing $L_1^-$ by $\overline{L_1^-}$ we may write $\epsilon_2'(0) = (\epsilon_2^{(1)},\epsilon_2^{(2)},\epsilon_2^{(3)})$. To first order the condition $[J'_1(t),J'_2(t)] = 0$ becomes the following compatibility conditions:
\begin{align}
\epsilon_2^{(1)} &= \epsilon_1^{(1)} \label{pccond1}\\
\epsilon_2^{(3)} &= \overline{\epsilon_1^{(3)}}. \label{pccond2}
\end{align}
We see that at the infinitesimal level a deformation of generalized K\"ahler structure has four components $\epsilon_1^{(1)},\epsilon_1^{(2)},\epsilon_1^{(3)},\epsilon_2^{(2)}$. Although the components so far are independent the integrability condition imposes relations between the derivatives of these components. The integrability condition is that $J_1',J_2'$ are integrable, which at the infinitesimal level translates to the conditions $d_{L_1} ( \epsilon_1^{(1)},\epsilon_1^{(2)},\epsilon_1^{(3)}) = 0$, $d_{L_2}( \epsilon_1^{(1)},\epsilon_2^{(2)},\overline{\epsilon_1^{(3)}}) = 0$.

Using the decomposition $L_1 = L_1^+ \oplus L_1^-$ we get projection maps $\wedge^k L_1^* \to \wedge^k (L_1^{\pm})^*$ which induce maps $\pi_1^{\pm} : H^k(L_1) \to H^k(L_1^{\pm} )$. Similarly we have maps $\pi_2^+ : H^k(L_2) \to H^k(L_1^+)$ and $\pi_2^- : H^k(L_2) \to H^k(\overline{L_1^-})$. Also observe that there is a natural antilinear isomorphism $H^k( \overline{L_1^-}) \to H^k( L_1^-)$. Let $\overline{\pi_2^-} : H^k(L_2) \to H^k(L_1^-)$ denote the composition of $\pi_2^-$ and this isomorphism.
\begin{proposition}\label{infdefgk}
Let $\rho_i \in H^2(L_i)$ for $i=1,2$ be the generalized Kodaira-Spencer classes corresponding to an infinitesimal deformation of generalized K\"ahler structure. Then
\begin{align}
\pi_1^+(\rho_1) &= \pi_2^+(\rho_2) \label{ccond1} \\
\pi_1^-(\rho_1) &= \overline{\pi_2^-}(\rho_2). \label{ccond2}
\end{align}
Moreover any pair $\rho_i \in H^2(L_i)$ satisfying (\ref{ccond1}),(\ref{ccond2}) can be realized as an infinitesimal deformation of generalized K\"ahler structure in the sense that there exists sections $\epsilon_1^{(1)},\epsilon_1^{(2)},\epsilon_1^{(3)},\epsilon_2^{(2)}$ of $\wedge^2 (L_1^+)^*$, $(L_1^+)^* \otimes (L_1^-)^*$, $\wedge^2 (L_1^-)^*$, $(L_1^+)^* \otimes (\overline{L_1^-})^*$ such that $d_{L_1} ( \epsilon_1^{(1)},\epsilon_1^{(2)},\epsilon_1^{(3)}) = 0$, $d_{L_2}( \epsilon_1^{(1)},\epsilon_2^{(2)},\overline{\epsilon_1^{(3)}}) = 0$ and $\rho_1 = [ ( \epsilon_1^{(1)},\epsilon_1^{(2)},\epsilon_1^{(3)})]$, $\rho_2 = [( \epsilon_1^{(1)},\epsilon_2^{(2)},\overline{\epsilon_1^{(3)}})]$.
\end{proposition}
\begin{proof}
Equations (\ref{ccond1}),(\ref{ccond2}) follow immediately from the corresponding pointwise conditions (\ref{pccond1}),(\ref{pccond2}). Now suppose $\rho_1,\rho_2$ satisfy (\ref{ccond1}),(\ref{ccond2}). Then we may write $\rho_i = [(\epsilon_i^{(1)},\epsilon_i^{(2)},\epsilon_i^{(3)}]$. The conditions on $\rho_1,\rho_2$ immediately implies that there exists $r \in \Gamma( (L_1^+)^*)$ and $s \in \Gamma( (\overline{L_1^-})^*)$ such that $\epsilon_2^{(1)} = \epsilon_1^{(1)} + d_{L_1^+} r$, $\epsilon_2^{(3)} = \overline{\epsilon_1^{(3)}} + d_{\overline{L_1^-}} (s)$. The result follows by replacing $(\epsilon_2^{(1)},\epsilon_2^{(2)},\epsilon_2^{(3)})$ by $(\epsilon_2^{(1)},\epsilon_2^{(2)},\epsilon_2^{(3)}) + d_{L_2}( r + s)$. 
\end{proof}
\begin{remark}
Proposition \ref{infdefgk} describes gives the conditions on the generalized Kodaira-Spencer classes associated to a deformation of generalized K\"ahler structure. It is possible however to have distinct deformations of generalized K\"ahler structure which have the same Kodaira-Spencer classes. For instance this happens with ordinary K\"ahler manifolds $(M,I,\omega)$ where we deform $\omega$ to $\omega + i \partial \overline{\partial} f$, for some function $f$.

Note also that we have not determined the conditions for an infinitesimal deformation of generalized K\"ahler structure to be realized by an actual family of generalized K\"ahler structures.
\end{remark}


\subsection{Variation of Hodge structure}\label{gkvhs}

We note that a smooth family $(\pi : M \to B,E,J_1,J_2)$ of generalized K\"ahler structures with compact fibers is a good family for $J_1$ and $J_2$ in the sense of Section \ref{vhs}, so we get smoothly varying Hodge decompositions corresponding to $J_1$ and $J_2$ as well as the associated period mappings. We have also seen that for a generalized K\"ahler manifold the Hodge decompositions for $J_1,J_2$ are compatible in that we get a bigraded decomposition in twisted cohomology:
\begin{equation}\label{gkhsd}
H^*_\mathbb{C} = \bigoplus_{r,s} H^{r,s}_{\overline{\delta_+}}
\end{equation}
which we can think of as a sort of bigraded Hodge structure. Here $H^*$ is the flat vector bundle associated to the twisted Gauss-Manin connection and $H^{r,s}_{\overline{\delta_+}}$ the subbundle represented by degree $(r,s)$ forms. Note that is actually is a smooth subbundle, since by elliptic semicontinuity the dimension of $H^{r,s}_{\overline{\delta_+}}$ is upper semicontinuous, but by (\ref{gkhsd}) the sum of the dimensions of the $H^{r,s}_{\overline{\delta_+}}$ is constant.\\

If $e$ is any vector field on the base and $\nabla$ the twisted Gauss-Manin connection we have by Proposition \ref{grtrans} that the covariant derivative $\nabla_e$ sends $H^{r,s}_{\overline{\delta}_+}$ to $H^{r-2}_{\overline{\partial}_1} \oplus H^r_{\overline{\partial}_1} \oplus H^{r+2}_{\overline{\partial}_1}$ and also to $H^{s-2}_{\overline{\partial}_2} \oplus H^s_{\overline{\partial}_2} \oplus H^{s+2}_{\overline{\partial}_2}$. While $\nabla_e$ is a differential operator, if we project out the $(r,s)$ component we get $\mathcal{C}^\infty(B)$-linear maps. Thus we have bundle maps $TB \otimes H^{r,s}_{\overline{\delta}_+} \to H^{r',s'}_{\overline{\delta}_+}$, where $(r',s') \neq (r,s)$. We illustrate this pictorially as the sum of the following terms:
\begin{equation*}\xymatrix{
H^{r-2,s+2}_{\overline{\delta}_+} & H^{r,s+2}_{\overline{\delta}_+} & H^{r+2,s+2}_{\overline{\delta}_+} \\
H^{r-2,s}_{\overline{\delta}_+} & H^{r,s}_{\overline{\delta}_+} \ar[r] \ar[l] \ar[u] \ar[d] \ar[ur] \ar[ul] \ar[dr] \ar[dl] & H^{r+2,s}_{\overline{\delta}_+} \\
H^{r-2,s-2}_{\overline{\delta}_+} & H^{r,s-2}_{\overline{\delta}_+} & H^{r+2,s-2}_{\overline{\delta}_+}
}
\end{equation*}
By Proposition \ref{grtrans} we know that these maps are essentially given by the Clifford action of the generalized Kodaira-Spencer classes on twisted cohomology. To explain this further we first note that there is a natural action $H^k(L_1^+) \otimes H^{r,s}_{\overline{\delta}_+} \to H^{r+k,s+k}_{\overline{\delta}_+}$ induced by the Clifford action. This is due to the identity $\overline{\delta}_+( a \beta) = (d_{L_1^+}a) \beta + (-1)^a a \overline{\delta}_+ \beta$, where $a \in \Gamma(\wedge^k (L_1^+)^*)$, $\beta \in \Gamma(U_{r,s})$. Similarly there is an action $H^k(L_1^-) \otimes H^{r,s}_{\overline{\delta}_-} \to H^{r+2,s-2}_{\overline{\delta}_-}$. Combining this with the isomorphism $H^{r,s}_{\overline{\delta}_-} \simeq H^{r,s}_{\overline{\delta}_+}$ defines an action of $H^k(L_1^-)$ on $H^{*,*}_{\overline{\delta}_+}$. Along similar lines $H^k( (\overline{L_1^+})^*)$ and $H^k( (\overline{L_1^-})^*)$ can also be made to act on $H^{*,*}_{\overline{\delta}_+}$.

Let $\rho_i : TB \to H^2(L_i)$ represent the generalized Kodaira-Spencer classes, $\rho_1^+$ the projection of $\rho_1$ to $H^2(L_1^+)$ and $\rho_1^-$ the projection to $H^2(L_1^-)$. Then for instance the map $TB \otimes H^{r,s}_{\overline{\delta}_+} \to H^{r+2,s+2}_{\overline{\delta}_+}$ is the Clifford action of $\rho_1^+$ and the map $TB \otimes H^{r,s}_{\overline{\delta}_+} \to H^{r+2,s-2}_{\overline{\delta}_+}$ is the Clifford action of $\rho_1^-$. Combining this with complex conjugation we account for four of the variations:
\begin{equation*}\xymatrix{
H^{r-2,s+2}_{\overline{\delta}_+} &  & H^{r+2,s+2}_{\overline{\delta}_+} \\
 & H^{r,s}_{\overline{\delta}_+} \ar[ur]^{\rho_1^+} \ar[ul]_{\overline{\rho_1^-}} \ar[dr]^{\rho_1^-} \ar[dl]_{\overline{\rho_1^+}} &  \\
H^{r-2,s-2}_{\overline{\delta}_+} &  & H^{r+2,s-2}_{\overline{\delta}_+}
}
\end{equation*}
The other four maps however are not so easy to describe cohomologically, but we can still describe them using representatives. As in Proposition \ref{infdefgk}, let $\rho_1 = [ ( \epsilon_1^{(1)},\epsilon_1^{(2)},\epsilon_1^{(3)})]$, $\rho_2 = [( \epsilon_1^{(1)},\epsilon_2^{(2)},\overline{\epsilon_1^{(3)}})]$. Then if $\beta$ is a $d_H$-closed degree $(r,s)$-form representing a class $[\beta] \in H^{r,s}_{\overline{\delta}_+}$ then $\nabla [\beta] \, ({\rm mod} \, H^{r,s}_{\overline{\delta}_+})$ is a sum of terms which can be arranged pictorially as follows:
\begin{equation*}\xymatrix{
\overline{\epsilon_1^{(3)}}\beta & \epsilon_2^{(2)}\beta & \epsilon_1^{(1)} \beta \\
\overline{\epsilon_1^{(2)}}\beta & \ast & \epsilon_1^{(2)}\beta \\
\overline{\epsilon_1^{(1)}}\beta & \overline{\epsilon_2^{(2)}}\beta & \epsilon_1^{(3)} \beta
}
\end{equation*}
Note that the sum of the above eight terms is $d_H$-closed modulo terms of degree $(r+1,s+1),(r+1,s-1),(r-1,s+1),(r-1,s-1)$. The degree $(r+2,s)$ component of $\nabla \beta$ for instance is given by $\epsilon_1^{(2)} \beta$. 

There is a special case in which all of these terms have a simple cohomological description. If the class $\rho_1 \in H^2(L_1)$ can be represented by a triple $\rho_1 = [(\tilde{\epsilon}_1^{(1)},\tilde{\epsilon}_1^{(2)},\tilde{\epsilon}_1^{(3)})]$, where $\tilde{\epsilon}_1^{(j)}$ is $d_{L_1}$-closed for $j=1,2,3$ then we obtain a decomposition $\rho_1 = \rho_1^{(1)} + \rho_1^{(2)} + \rho_1^{(3)}$, where $\rho_1^{(j)} = [\tilde{\epsilon}_1^{(j)}] \in H^2(L_1)$. If this is possible it is straightforward to see that we must have that $\rho_1^{(1)} = \rho_1^+$, $\rho_1^{(3)} = \rho_1^-$ are the projections from $H^2(L_1)$ to $H^2(L_1^+),H^2(L_1^-)$ respectively. Also the class $\rho_1^{(2)}$ has the property that its Clifford action sends $H^{r,s}_{\overline{\delta}_+}$ to $H^{r+2,s}_{\overline{\delta}_+}$ and then the component $TB \otimes H^{r,s}_{\overline{\delta}_+} \to H^{r+2,s}_{\overline{\delta}_+}$ of the variation of Hodge structure is just the Clifford action by the class $\rho_1^{(2)}$. However we suspect that such a decomposition $\rho_1 = [(\tilde{\epsilon}_1^{(1)},\tilde{\epsilon}_1^{(2)},\tilde{\epsilon}_1^{(3)})]$ need not be true in general.


\bibliographystyle{amsplain}

\end{document}